\documentclass[12pt,reqno]{amsart} %&latex209 
\usepackage{amsmath} 

\usepackage{amsthm}
\usepackage{amssymb}
\usepackage{graphicx}
\usepackage{latexsym}

\numberwithin{equation}{section}
\newtheorem{thm}{Theorem}[section]
\newtheorem{prop}[thm]{Proposition}
\newtheorem{lem}[thm]{Lemma}

\newtheorem{rem}{Remark}[section]
\newtheorem{defn}{Definition}
 
\newcommand{\laplacian}{\Delta}
\newcommand{\BBB}{\mathbb}
\newcommand{\R}{{\BBB R}}

\newcommand{\LR}[1]{{\langle {#1} \rangle }}
\newcommand{\lec}{{\ \lesssim \ }}
\newcommand{\gec}{{\ \gtrsim \ }}

\newcommand{\cross}{\times}

\newcommand{\e}{\varepsilon}

\newcommand{\p}{\partial}

\newcommand{\supp}{\operatorname{supp}}

\newcommand{\I}{\infty}

\newcommand{\EQ}[1]{\begin{equation} \begin{split} #1
 \end{split} \end{equation}}
\newcommand{\EQS}[1]{\begin{align} #1 \end{align}}
\newcommand{\EQQS}[1]{\begin{align*} #1 \end{align*}}

\newcommand{\ol}{\overline}

\newcommand{\F}{\mathcal{F}}
\newcommand{\1}{{\mathbf 1}}

\title[LWP for the degenerate Zakharov system]{Local well-posedness of the Cauchy problem for the degenerate Zakharov system}

\author[I. Kato]{Isao Kato %\\ 
}  
\address[Isao Kato]{Department of Mathematics, Kyoto University,
Kitashirakawa Oiwake-cho, Sakyo-ku, Kyoto, 606-8502, Japan} 
\email[Isao Kato]{kato.isao.23n@st.kyoto-u.ac.jp}

\subjclass[2010]{35Q55, 35A01, 35A02}
\keywords{Cauchy problem, well-posedness, low regularity, $U^2, V^2$ spaces} 

\begin{document}

\begin{abstract}
The aim of this paper is to investigate well-posedness of the Cauchy problem for the degenerate Zakharov system.     
Local well-posedness holds for anisotropic Sobolev data by applying $U^2, V^2$ type spaces.  
We give the Schr\"odinger initial data $H^{s_k,\, s'}$ and the wave data $H^{s_l,\, s'}$ 
where $s_k > (d-1)/2,\, s_l > (d-2)/2,\, s_k - s_l = 1/2$ and $s' > 1/2$.  
\end{abstract} 
\maketitle 
\setcounter{page}{001}

%%%%%%%%%%%%%%%%%%%%%%%%%%%%%%%%%%%%%%%%%%%%%%%%%%%%%%%%%%%%%%%%%%%%%%%%%%%%%
%%%%%%%%%%%%%%%%%%%%%%%%%%%%%%%%%  Section 1  %%%%%%%%%%%%%%%%%%%%%%%%%%%%%%%
%%%%%%%%%%%%%%%%%%%%%%%%%%%%%%%%%%%%%%%%%%%%%%%%%%%%%%%%%%%%%%%%%%%%%%%%%%%%%

\section{Introduction}
We consider the Cauchy problem for the degenerate Zakharov system: 
\EQS{
 \begin{cases}
  i(\partial_t u + \partial_{x_d} u) + \Delta_{\perp} u = nu, \qquad \qquad (t,x) \in [0,T] \cross \R^d, \\
  \partial_t^2 n - \Delta_{\perp}n = \Delta_{\perp} |u|^2, \hspace{6em} \ (t,x) \in [0,T] \times \mathbb{R}^d, \\
  (u, n, \partial_t n)|_{t=0} = (u_0, n_0, n_1),                                                                                                                          
                  \label{degenerateZ}
 \end{cases}  
} 
where $d \ge 2,\, \Delta_{\perp} = \displaystyle \sum_{i=1}^{d-1}\p_{x_i}^2, u$ is complex valued function and $n$ is real valued function. 
For three spatial dimension, \eqref{degenerateZ} describes the laser propagation when the paraxial approximation is used and the effect of the group velocity is negligible \cite{SulSul}, \cite{LPS}, \cite{BL}. 
For the Cauchy problem for the Zakharov system   
\EQS{
 \begin{cases}
  i\partial_t u + \Delta u = nu,\qquad \qquad \ \ \ (t, x) \in [0, T] \times \mathbb{R}^d, \\
  \partial_t^2 n - \Delta n = \Delta |u|^2,\hspace{4em} (t, x) \in [0, T] \times \mathbb{R}^d, \\ 
  (u, n, \partial_t n)|_{t=0} = (u_0, n_0, n_1),   
 \end{cases}
                 \label{Zakharov}
}
where $\Delta = \displaystyle \sum_{i=1}^d \partial_{x_i}^2$, 
the well-posedness for \eqref{Zakharov} is well studied with low regularity initial data $(u_0, n_0, |\nabla _x|^{-1} n_1) \in H^k(\mathbb{R}^d) \times H^l(\mathbb{R}^d) \times H^l(\mathbb{R}^d)$, for instance 
Ginibre, Tsutsumi and Velo \cite{GTV} for all spatial dimensions, 
Bejenaru, Herr, Holmer and Tataru \cite{BHHT} for $d=2$,    
Bejenaru and Herr \cite{BH} for $d=3$ and 
Bejenaru, Guo, Herr and Nakanishi \cite{BGHN}, 
Candy, Herr and Nakanishi \cite{CHN} or the author and Tsugawa \cite{KaT} for $d \ge 4$ case.    
On the other hand, the degenerate Zakharov system \eqref{degenerateZ} with low regularity initial data 
is much less considered than that for the Zakharov system. 
M. Colin and T. Colin \cite{CC} derived the Cauchy problem for \eqref{degenerateZ} in three spatial dimension.  
Later Linares, Ponce and Saut \cite{LPS} proved local well-posedness in a suitable function space by deriving the local smoothing estimate and the maximal function estimate. 
More precisely, the lowest regularity initial data is $(u_0, n_0, n_1) \in \tilde{H}^5(\mathbb{R}^3) \times H^5(\mathbb{R}^3) \times H^4(\mathbb{R}^3)$ and  $\partial_{x_3} n_1 \in H^4(\mathbb{R}^3)$, where 
\EQQS{
 \tilde{H}^s(\mathbb{R}^3) = \{ f \in H^s(\mathbb{R}^3)\, |\, \partial_{x_1}^{1/2}\partial^{\alpha}f,\, \partial_{x_2}^{1/2}\partial^{\alpha}f \in L^2(\mathbb{R}^3), \, |\alpha| \le s,\, \alpha \in (\mathbb{Z}_{\ge 0})^3 \}. 
}  

Barros and Linares \cite{BL} obtained local well-posedness for initial data $(u_0, n_0, n_1) \in \tilde{H}^2(\mathbb{R}^3) \times H^2(\mathbb{R}^3) \times H^1(\mathbb{R}^3)$ and $\partial_{x_3}n_1 \in H^1(\mathbb{R}^3)$. 
The key is improving the regularity of the maximal function estimate \eqref{maximal_function_est} and deriving the Strichartz estimate for the Schr\"odinger equation \eqref{Strichartz_Estimate} below. 
We may see lack of dispersion in $x_3$ variable in \eqref{Strichartz_Estimate}. 
\EQS{
 &\| S(t)f\|_{L^2_{x_1} L^{\infty}_{x_2,\, x_3,\, T}} \le c(T,s)\|f\|_{H^s(\mathbb{R}^3)},\quad s > 3/2,  \label{maximal_function_est} \\ 
 &\| S(t)g\|_{L^q_t L^r_{x_1,\, x_2} L^2_{x_3}} \le c\|g\|_{L^2_{x_1,\, x_2,\, x_3}},   \label{Strichartz_Estimate}
}
where $2/q = 1-2/r,\, 2 \le r < \infty,\, c(T,s)$ and $c$ are constants, $S(t)$ is the free Schr\"odinger operator, see section 2 for precise definition. 
   
The aim of this paper is to prove local well-posedness by $U^2, V^2$ type spaces in all spatial dimensions  $(d \ge 2)$. 
The nonlinear estimate namely Proposition \ref{BE} plays an important role in this paper. 
To prove Proposition \ref{BE}, we remark that (i) we use the anisotropic Sobolev data $H^{s,\, s'}(\mathbb{R}^d)$, (ii) we construct the solution for the Schr\"odinger equation in the intersection space between $V^2$ based space and the spaces associated to the local smoothing and the maximal function estimate, (iii) we consider $s_k - s_l = 1/2$, where $s_k, s_l$ denotes regularity of the initial data for the Schr\"odinger and the wave equation respectively. 
Concerning (i), the Schr\"odinger equation in \eqref{degenerateZ} does not have a dispersion in the $x_d$ variable, 
it seems hard to find smoothing property including the $x_d$ direction. 
Indeed, the local smoothing estimate Proposition \ref{local_smoothing} \eqref{local_smoothing_est} does not have 
a smoothing with respect to the $x_d$ variable.  
On the other hand, from the following modulation estimate \eqref{modulation_est} we can recover almost $1/2$ regularity with respect to $\Xi_1 = (\xi_1,\, \xi_1') \in \mathbb{R}^{d-1} \times  \mathbb{R}$.  
Let $\tau_3=\tau_1-\tau_2,\ \xi_3 = \xi_1 - \xi_2,\ |\xi_1| \gg |\xi_2|$ and $|\xi_1|^2 \gg \max \{ |\xi_1'|, |\xi_2'|\}$. Then it holds that 
\EQS{ 
 \max \big\{ \big|\tau_1 + |\xi_1|^2 +\xi_1' \big|, \big|\tau_2 + |\xi_2|^2 +\xi_2' \big|, \big|\tau_3 \pm |\xi_3| \big| \big\} 
  \gec |\Xi_1|.   
                                      \label{modulation_est}
} 
However applicable case is limited hence it is natural to use anisotropic Sobolev data $H^{s,\, s'}(\mathbb{R}^d)$ for \eqref{degenerateZ}.   
In Theorem \ref{mth}, we take $1/2+\varepsilon'$ regularity with respect to the $x_d$ variable for any $\varepsilon' > 0$. 
This regularity comes from the natural Sobolev embedding $H^{1/2+\varepsilon'}(\mathbb{R}) \hookrightarrow L^{\infty}(\mathbb{R})$. 
For (ii), we cannot replace $V^2$ by $L^{\infty}_t L^2_x$ since one of the nonlinear estimate for the Schr\"odinger equation, for $J_{4,\, 3,\, 1}$ in the proof of Proposition \ref{BE} we have to apply $U^2_S \hookrightarrow L^{\infty}_t L^2_x$ since the frequency $N_1$ of the wave equation is low.    
Moreover it seems difficult to control in the framework of $X^{s,\, b}$ defined by the Fourier restriction norm method.  
For the estimate of $J_{1,\, 1}$ in the proof of Proposition \ref{BE}, we need to gain $1/2$ regularity by the local smoothing estimate, however we run short of $\varepsilon$ regularity with time variable by applying the $X^{s,\, b}$ space. 
To solve the problem, we use the $U^2, V^2$ type spaces which are enhanced $X^{s,\, b}$ spaces.     
Now we mention (iii), namely $s_k - s_l = 1/2$ balances the two nonlinear estimates for the Sch\"odinger and the wave equations.  
Once we take $s_k, s_l$ satisfying $s_k - s_l > 1/2$, we may find that the nonlinear estimate for the Schr\"odinger equation, for instance $J_{1,\, 1}$ is hard to estimate since we need to gain more than $1/2$ regularity. 
By applying the local smoothing estimate, we only gain $1/2$ regularity hence it is not sufficient to derive the estimate for $J_{1,\, 1}$.  
We note that $J_{1,\, 1}$ contains resonant region and cannot apply the modulation estimate \eqref{modulation_est} there. 
Also, the bilinear Strichartz estimate Proposition \ref{x01/2} does not gain enough regularity to overcome the difficulty. 
The case $s_k - s_l < 1/2$ contains similar problem, hence $s_k-s_l=1/2$ is important in our analysis.   

If we transform $n_{\pm} := n \pm i\, \omega^{-1}\partial_t n, \omega := (-\laplacian_{\perp} )^{1/2}$, then \eqref{degenerateZ} is equivalent to the following:  
\EQS{
 \begin{cases}
  i(\partial_t  + \partial_{x_d} )u + \Delta_{\perp}u  = (1/2)(n_+ + n_-)u, \qquad \,   (t,x) \in [0,T] \times \mathbb{R}^d, \\
  (i\partial_t \mp \omega )n_{\pm} 
    = \pm \omega \, |u|^2, \hspace{9.5em} (t,x) \in [0,T] \times \mathbb{R}^d, \\
  (u, n_{\pm})|_{t=0} = (u_0, n_{\pm 0})                                                                                                                         
 \label{degenerateZ'} 
 \end{cases}  
}
The main result is as follows. 
\begin{thm}  \label{mth} 
Let $d \ge 2, s>(d-1)/2, s'>1/2$ and assume the initial data $(u_0,n_{\pm 0}) \in H^{s,\, s'}(\mathbb{R}^d) \times H^{s-1/2,\, s'}(\mathbb{R}^d)$. 
Then, \eqref{degenerateZ'} is locally well-posed in $H^{s,\, s'}(\mathbb{R}^d) \times H^{s-1/2,\, s'}(\mathbb{R}^d)$.   
\end{thm} 

Now we mention the critical exponent for \eqref{degenerateZ'}. 
The scaling between the first and the second equation of \eqref{degenerateZ'} is different, hence there is no critical index for \eqref{degenerateZ'} in the strict sense. 
However, if we neglect $\mp \omega n_{\pm}$ in \eqref{degenerateZ'}, then under the scaling 
\EQQS{
 &u_{\lambda}(t, \tilde{x}, x_d) := \lambda^{-3/2}u( \lambda^{-2}t, \lambda^{-1}\tilde{x}, \lambda^{-2}x_d),\quad \ \, \tilde{x} = (x_1, ... , x_{d-1}) \in \mathbb{R}^{d-1},\, \lambda > 0, \\ 
 &n_{\pm \lambda}(t, \tilde{x}, x_d) := \lambda^{-2}n_{\pm}( \lambda^{-2}t, \lambda^{-1}\tilde{x}, \lambda^{-2}x_d),\quad \tilde{x} = (x_1, ... , x_{d-1}) \in \mathbb{R}^{d-1},\, \lambda > 0,  
}
the critical exponent is $k+2k'=(d-2)/2$ and $l+2l'=(d-3)/2$, 
where $(u_0, n_{\pm 0}) \in H^{k,\, k'}(\mathbb{R}^d) \times H^{l,\, l'}(\mathbb{R}^d)$. 
Hence, we expect that Theorem \ref{mth} is not optimal. 
However our result improves almost $1$ derivative than \cite{BL} for $d=3$. 

Finally, we mention the conservation laws for \eqref{degenerateZ}. 
$\|u(t)\|_{L^2_x}$ and 
\EQQS{ 
 \|(-\Delta_{\perp})^{1/2}u\|_{L^2_x}^2 + \text{Im}\, \int \bar{u}\, \partial_{x_d}u\, dx + \int n|u|^2\, dx + \frac{1}{2}\|n\|_{L^2_x}^2 + \frac{1}{2}\|(-\Delta_{\perp})^{-1/2}\partial_t n\|_{L^2_x}^2 
}
are conserved for \eqref{degenerateZ}.   
There is no global well-posedness result for \eqref{degenerateZ}, 
hence by using these quantities we expect obtaining global well-posedness 
in future. 

In section 2, we introduce notations, lemmas and solution spaces $X_S, Z^{s_l,\, s'}_{W_{\pm}}$.  
In section 3, we show the crucial nonlinear 
%bilinear 
estimate to obtain Theorem \ref{mth}. 
%In section 4, we give the proof of the Theorem \ref{mth}.    
Finally in section 4, we derive the bilinear Strichartz estimate for future reference.

%%%%%%%%%%%%%%%%%%%%%%%%%%%%%%%%%%%%%%%%%%%%%%%%%%%%%%%%%%%%%%%%%%%%%%%%%%%%%
%%%%%%%%%%%%%%%%%%%%%%%%%%%%%%%%%  Section 2  %%%%%%%%%%%%%%%%%%%%%%%%%%%%%%%
%%%%%%%%%%%%%%%%%%%%%%%%%%%%%%%%%%%%%%%%%%%%%%%%%%%%%%%%%%%%%%%%%%%%%%%%%%%%%
\section{Notations and Preliminary Lemmas}
In this section, we prepare some lemmas, propositions and notations to prove the main theorem. 
$A \lec B$ means that there exists $C>0$ such that $A \le CB$.  
Also, $A \sim B$ means $A \lec B$ and $B \lec A$.  
Let $a \pm := a \pm \e$ for sufficiently small $\e > 0$ 
and 
$\tilde{x} = (x_1, ... , x_{d-1}) \in \mathbb{R}^{d-1}$. 
Let $u = u(t,x).\ \F_t u,\ \F_x u$ denote the Fourier transform of $u$ in time, space, respectively.  
We denote $\|f\|_2 = \|f\|_{L^2_{t,\, x}}$ and we define $H^{s,\, s'}(\mathbb{R}^d)$ with the norm 
\EQQS{ 
 \|g\|_{H^{s,\, s'}} := \Bigl( \int_{\mathbb{R}^d} \LR{\xi}^{2s} \LR{\xi'}^{2s'} |\mathcal{F}_x g(\xi, \xi')|^2 d\xi d\xi' \Bigr)^{1/2}
} 
for $\xi \in \mathbb{R}^{d-1},\, \xi' \in \mathbb{R},\,  
\LR{\cdot} := (1+|\cdot|^2)^{1/2}$. 
 
%%%%%%%%%%%%%%%%%%%%%%%%%%%%%%%%%%
%%%%%  Notation of operators %%%%%
%%%%%%%%%%%%%%%%%%%%%%%%%%%%%%%%%%
Let $\{ \mathcal{F}_{y}^{-1}[\varphi_n](y)\}_{n \in \mathbb{Z}} \subset \mathcal{S}(\mathbb{R})$ be the Littlewood-Paley decomposition, that is to say 
\EQQS{
\begin{cases}
 \varphi(\eta) \ge 0, \\
 \supp \varphi(\eta) = \{ \eta \,|\, 2^{-1} \le |\eta| \le 2\}, 
\end{cases}
}
\EQQS{
 \varphi_{n}(\eta) := \varphi(2^{-n}\eta),\ \sum_{n=-\I}^{\I}\varphi_{n}(\eta)=1\ (\, \eta \neq 0),\ 
 \psi(\eta) := 1 - \sum_{n=0}^{\I}\varphi_{n}(\eta). 
} 
Let $N,\, N' \in 2^{\mathbb{Z}}$ be dyadic numbers.  
$P_{N,\, N'}$ and $P_0 = P_{0,\, 0}$ denote 
\EQQS{
 &\mathcal{F}_{x}[P_{N,\, N'} f](\xi,\, \xi') := \varphi(|\xi|/N) \varphi(|\xi'|/N')  \mathcal{F}_x[f](\xi,\, \xi') = \varphi_n(|\xi|) \varphi_{n'}(|\xi'|) \mathcal{F}_x[f](\xi,\, \xi'), \\ 
 &\mathcal{F}_{x}[P_0 f](\xi,\, \xi') := \psi(|\xi|) \psi(|\xi'|) \mathcal{F}_x[f](\xi,\, \xi').   
} 

Let $S(t)={\rm exp}\{ t(i\laplacian_{\perp}-\partial_{x_d}) \} : L^2_x \to L^2_x$ be the linear operator associated to the Schr\"odinger equation: 
\EQQS{ 
 i(\partial_t + \partial_{x_d})u + \Delta_{\perp}u = 0. 
}
Namely, 
$\mathcal{F}_x [S(t)f](\xi, \xi') := e^{-it(|\xi|^2 + \xi')}\F_x f(\xi, \xi')$ for $\xi \in \mathbb{R}^{d-1},\, \xi' \in \mathbb{R}$.  
Similarly, we define the wave unitary operator $W_{\pm}(t):=\text{exp}\{ \mp it\, \omega \} : L^2_x \to L^2_x$ such that $\mathcal{F}_x  [W_{\pm}(t)g](\xi, \xi'):=e^{\pm it|\xi|}\mathcal{F}_x g(\xi, \xi')$ for $\xi \in \mathbb{R}^{d-1},\, \xi' \in \mathbb{R}$. 

%%%%% Definition of U^p, V^p %%%%% 
Let $\mathcal{Z}$ be the set of finite partitions $-\I = t_0 < t_1 < \cdots < t_K = \I$.  %and let $\mathcal{Z}_0$ be the set of finite partitions $-\I < t_0 < t_1 < \cdots < t_K \le \I$.   
\begin{defn}
%Let $1 \le p < \infty$. 
For $\{t_k\}_{k=0}^K \in \mathcal{Z}$ and $\{ \phi_k\}_{k=0}^{K-1}\subset L^2_x$ with   
$\sum_{k=0}^{K-1} \|\phi_k \|_{L^2_x}^2 = 1$ and $\phi_0=0$, we call the function $a : \mathbb{R} \to L^2_x$ given by 
\EQQS{
 a = \sum_{k=1}^K \1_{[t_{k-1},\, t_k)}\phi_{k-1} 
}
a $U^2$-atom. 
Furthermore, we define the atomic space 
\EQQS{
 U^2 := \Bigl\{ u=\sum_{j=1}^{\I}\lambda_j a_j \, \Bigl| \, a_j : U^2 \text{-atom} , \lambda_j \in \mathbb{C} \ such\ that\ \sum_{j=1}^{\I}|
          \lambda_j| < \I \Bigr\}, 
}
with norm 
\EQQS{
 \|u\|_{U^2} := \inf \Bigl\{ \sum_{j=1}^{\I}|\lambda_j| \, \Bigl| \, u=\sum_{j=1}^{\I}\lambda_j a_j, \lambda_j \in \mathbb{C}, a_j : U^2 \text{-atom} \Bigr\}.
} 
\end{defn}
\begin{defn} \label{def_of_V}
%Let $1 \le p < \I$.  
We define $V^2$ as the normed space of all functions $v:\mathbb{R} \to L^2_x$ such that 
$\lim_{t \to \pm \I}v(t)$ exist and for which the norm 
\EQQS{
\|v\|_{V^2} := \sup_{\{ t_k\}_{k=0}^K \in \mathcal{Z}} \Bigl(\sum_{k=1}^K\| v(t_k)-v(t_{k-1})\|_{L^2_x}^2 \Bigr)^{1/2}   
} 
is finite, where we use the convention that $v(-\I):=\lim_{t \to -\I}v(t)$ and $v(\I):=0.$
Note that $v(\infty)$ does not necessarily coincide with the limit at $\infty$.
Likewise, let $V_-^2$ denote the closed subspace of all $v \in V^2$ with $\lim_{t \to -\I}v(t)=0$.  
\end{defn}
%For the definitions of $V^2$ and $V^2_-$, see \cite{HHK2}. 
\begin{defn}
For $A = S$ or $W_{\pm}$, we define 
\EQQS{
&\hspace{-55mm} {\rm(\hspace{.18em}i\hspace{.18em})}\ U^2_A = A(\cdot)U^2 \ \text{with}\ \text{norm}\ \|u\|_{U^2_A} = \|A(-\cdot)u\|_{U^2}, \\ 
&\hspace{-55mm} {\rm(\hspace{.08em}ii\hspace{.08em})}\ V^2_A = A(\cdot)V^2\ \text{with}\ \text{norm}\ \|u\|_{V^2_A} = \|A(-\cdot)u\|_{V^2}.  
}  
\end{defn}
See \cite{HHK}, \cite{HHK2} for more detail. 
%%%%%
\begin{defn}
 For a Hilbert space $H$ and a Banach space $E \subset C(\mathbb{R}; H)$, we define $B_r(H) := \{ f \in H \, |\, \|f\|_H \le r \}$ and 
\EQQS{ 
 E([0, T]) := \{ u \in C([0, T]; H)\, |\, \exists \tilde{u} \in E, \tilde{u}(t) = u(t),t \in [0, T] \},  
} 
endowed with the norm $\|u\|_{E([0, T])} = \inf \{ \|\tilde{u}\|_E \, |\, \tilde{u}(t) = u(t), t \in [0, T] \}$. 
\end{defn}
Hereafter, we denote $E$ instead of $E([0, T])$ for brevity. 
We introduce the solution spaces $X_S$ and $Z^{s_l,\, s'}_{W_{\pm}}$ below. 
\begin{defn}
Let $s_k > (d-1)/2, s_l > (d-2)/2, s' > 1/2$ and $\varepsilon > 0$ is sufficiently small such that $s_k - (d-1)/2 - \varepsilon > 0$. 
We define the function spaces $X_S, Y^{s_k,\, s'}_S, Y^{s_l,\, s'}_{W_{\pm}}, Z^{s_l,\, s'}_{W_{\pm}}$ as follows: 
\EQQS{ 
 &X_S := \{ u \in C([0,\, T] ; H^{s_k,\, s'}(\mathbb{R}^d)) \, |\, \|u\|_{X_S} < \infty \}, \\
 &Y^{s_k,\, s'}_S := \{ u \in C([0,\, T] ; H^{s_k,\, s'}(\mathbb{R}^d) \, |\, \|u\|_{Y^{s_k,\, s'}_S} < \infty \}, \\  
 &Z^{s_l,\, s'}_{W_{\pm}} := \{ n \in C([0,\, T] ; H^{s_l,\, s'}(\mathbb{R}^d)) \, |\, \|n\|_{Z^{s_l,\, s'}_{W_{\pm}}} < \infty \}, 
} 
where 
\EQQS{ 
 \|u\|_{X_S} = J + K + M,\qquad 
 \|u\|_{Y^{s_k,\, s'}_S} = J,\qquad 
 \|n\|_{Z^{s_l,\, s'}_{W_{\pm}}} = R,  
} 
\EQQS{ 
 J &:= \| P_0 u\|_{V^2_S} + 
 \Bigl( \sum_{N,\, N' \ge 1} N^{2s_k} N'^{2s'} \| P_{N,\, N'} u\|_{V^2_S}^2 \Bigr)^{1/2}, \\ 
 K &:= \max_{1 \le i \le d-1} \Bigl( \| P_0 u\|_{L^{\infty}_{x_i} L^2_{x_1, ... , x_{i-1}, x_{i+1}, ... , x_d, t}} \\ 
   &\qquad + \Bigl( \sum_{N,\, N' \ge 1} N^{2(s_k + 1/2)} N'^{2s'} \|P_{N,\, N'} u\|_{L^{\infty}_{x_i} L^2_{x_1, ... , x_{i-1}, x_{i+1}, ... , x_d, t}}^2  \Bigr)^{1/2} \Bigr), \\  
 M &:= \max_{1 \le i \le d-1} \Bigl( \| P_0 u\|_{L^2_{x_i} L^{\infty}_{x_1, ... , x_{i-1}, x_{i+1}, ... , x_{d-1}, t} L^2_{x_d}} \\  
   &\qquad + \Bigl( \sum_{N,\, N' \ge 1} N^{2(s_k-(d-1)/2-\varepsilon)} N'^{2s'} \|P_{N,\, N'}u\|_{L^2_{x_i} L^{\infty}_{x_1, ... , x_{i-1}, x_{i+1}, ... , x_{d-1}, t} L^2_{x_d}}^2 \Bigr)^{1/2} \Bigr), \\ 
 R &:= \| P_0 n\|_{U^2_{W_{\pm}}} + \Bigl( \sum_{N,\, N' \ge 1} N^{2s_l}  N'^{2s'} \|P_{N,\, N'} n\|_{U^2_{W_{\pm}}}^2 \Bigr)^{1/2}. 
} 
Similarly we define $Y^{s_l,\, s'}_{W_{\pm}}$ as follows. 
\EQQS{  
 &Y^{s_l,\, s'}_{W_{\pm}} := \{ n \in C([0,\, T] ; H^{s_l,\, s'}(\mathbb{R}^d) \, |\, \|n\|_{Y^{s_l,\, s'}_{W_{\pm}}} < \infty \}, \\ 
 &\|n\|_{Y^{s_l,\, s'}_{W_{\pm}}} := \|P_0 n\|_{V^2_{W_{\pm}}} + \Bigl( \sum_{N,\, N' \ge 1} N^{2s_k} N'^{2s'} \|P_{N,\, N'}n\|_{V^2_{W_{\pm}}}^2 \Bigr)^{1/2}.       
}
\end{defn} 

The integral equations for \eqref{degenerateZ'} are as follows.   
\EQQS{
 &I_{T,\, S}(n_+, n_-, u) = S(t)u_0 - \frac{i}{2} \int_0^t S(t-t') (n_+(t') + n_-(t')) u(t')\, dt', \\
 &I_{T,\, W_{\pm}}(u, v) = W_{\pm}(t)n_{\pm 0} \pm \int_0^t W_{\pm}(t-t')\, \omega \bigl( u(t') (\overline{v(t')})\bigr)\, dt'. 
}

The proof of the following proposition for $d=3$ can be found in \cite{LPS}. 
Analogously we can show \eqref{local_smoothing_est}.  

\begin{prop} \label{local_smoothing} 
For $d \ge 2,\, f \in L^2$, it holds that 
\EQS{ 
 &\| S(t) P_0 f\|_{L^{\infty}_{x_1} L^2_{x_2, ... , x_d, t}} 
      \lec \|P_0 f\|_2, \label{low_frequency_local_smoothing_est} \\ 
 &\| D_{x_1}^{1/2}\, S(t) f\|_{L^{\infty}_{x_1} L^2_{x_2, ... , x_d, t}} 
      \lec \|f\|_2, 
                   \label{local_smoothing_est}
}
where $\mathcal{F}_{x_1} [D_{x_1}^{1/2}\, u] = |\xi_1|^{1/2} \mathcal{F}_{x_1} [u]$. 
\eqref{low_frequency_local_smoothing_est}, \eqref{local_smoothing_est} hold exchanging $x_1$ for $x_i,\, i \in \{ 2, ... , $ \\ 
$d-1\}$.  
\end{prop}
\begin{rem} \label{loc_smoothing_U^2}
 From \eqref{local_smoothing_est}, we have $\|D^{1/2}_{x_1}f\|_{L^{\infty}_{x_1} L^2_{x_2, ... , x_d, t}} \lec \|f\|_{U^2_S}$. 
\end{rem}

\begin{prop} \label{maximal_function}
For $d \ge 2,\, s > (d-1)/2$ and $g \in L^2,\, h \in H^s_{\tilde{x}} L^2_{x_d}$, it holds that 
\EQS{
 &\| S(t) P_0 g\|_{L^2_{x_1} L^{\infty}_{x_2, ... , x_{d-1}, t} L^2_{x_d}} \lec \|P_0 g\|_2, 
      \label{low_frequency_maximal_function_estimate}  \\   
 &\| S(t) h\|_{L^2_{x_1} L^{\infty}_{x_2, ... , x_{d-1}, t} L^2_{x_d}} \lec \|h\|_{H^s_{\tilde{x}} L^2_{x_d}}. 
            \label{maximal_function_estimate}
} 
The above estimates hold exchanging $x_1$ for $x_i,\, i \in \{ 2, ... , d-1\}$. 
\end{prop} 
 
Proposition \ref{maximal_function} 
is proved by the same manner as $d-1$ dimensional maximal function estimate for the Schr\"odinger equation, hence we omit the proof.

\begin{prop} \label{inhomo_L^infty_tL^2_x}
For $d \ge 2, f \in \mathcal{S}(\mathbb{R}^{d+1})$, it holds that 
\EQS{
 &\Bigl\| \int_0^t S(t-t') P_0 f(t')\, dt' \Bigr\|_{L^{\infty}_t L^2_x} 
       \lec \|P_0 f\|_{L^1_{x_1} L^2_{x_2, ... , x_d, t}}, 
                    \label{lowfreqconti} \\ 
 &\Bigl\| D^{1/2}_{x_1} \int_0^t S(t-t') f(t')\, dt' \Bigr\|_{L^{\infty}_t L^2_x} 
       \lec \|f\|_{L^1_{x_1} L^2_{x_2, ... , x_d, t}}.  
                    \label{conti} 
}
The above estimates hold exchanging $x_1$ for $x_i,\, i \in \{ 2, ... , d-1\}$. 
\end{prop}
\begin{proof}
The dual estimate of Proposition \ref{local_smoothing}  \eqref{low_frequency_local_smoothing_est} is 
\EQQS{
 \Bigl\| \int_{-\infty}^{\infty} S(-t') P_0 f(t') dt' \Bigr\|_{L^2_x} \lec \|P_0 f\|_{L^1_{x_1} L^2}. 
} 
Since $S(t)$ is the unitary operator on $L^2_x$, 
\EQQS{
 \Bigl\| \int_{-\infty}^{\infty} S(t-t') P_0 f(t') dt' \Bigr\|_{L^2_x} \lec \|P_0 f\|_{L^1_{x_1} L^2}. 
}
Replacing $f(t')$ by $\1_{[0,\, t]}(t') f(t')$, then take $L^{\infty}_t$ norm and \eqref{lowfreqconti} follows. 
\eqref{conti} follows from \cite{LPS} Proposition 2.1 for $d=3$. 
The other cases are treated similarly, hence we omit the proof. 
\end{proof}

\begin{prop} \label{inhomo_local_smoothing}
For $d \ge 2, f \in \mathcal{S}(\mathbb{R}^{d+1})$, the following estimates hold:  
\EQS{
 &\Bigl\| \int_0^t S(t-t') P_0 f(t')\, dt'\Bigr\|_{L^{\infty}_{x_1} L^2_{x_2, ... , x_d, t}} 
       \lec \|P_0 f\|_{L^1_{x_1} L^2_{x_2, ... , x_d, t}}, 
                    \label{lowfreqsmoothing} \\ 
 &\Bigl\| \partial_{x_1} \int_0^t S(t-t') f(t')\, dt'\Bigr\|_{L^{\infty}_{x_1} L^2_{x_2, ... , x_d, t}} 
       \lec \|f\|_{L^1_{x_1} L^2_{x_2, ... , x_d, t}}, 
                    \label{smoothing}  
} 
The above estimates hold exchanging $x_1$ for $x_i,\, i \in \{ 2, ... , d-1\}$.
\end{prop} 
\begin{proof}   
From the Sobolev inequality and \eqref{lowfreqconti}, we have 
\EQQS{
 \Bigl\| \int_0^t S(t-t') P_0 f(t') dt' \Bigr\|_{L^{\infty}_{x_1} L^2} &\lec T^{1/2} \Bigl\| |\nabla_{x_1}|^{1/2+} \int_0^t S(t-t') P_0 f(t') dt' \Bigr\|_{L^{\infty}_t L^2_x} \\  
 &\lec \Bigl\| \int_0^t S(t-t') P_0 f(t') dt' \Bigr\|_{L^{\infty}_t L^2_x} \\ 
 &\lec \|P_0 f\|_{L^1_{x_1} L^2}. 
}  
Thus \eqref{lowfreqsmoothing} follows. 
The proof of \eqref{smoothing} for $d=3$ can be found in \cite{LPS} Proposition 2.1. Analogously we have \eqref{smoothing}. 
\end{proof}

\begin{prop} \label{inhomo_maximal_fcn} 
 Let $d \ge 2,\, s > (d-1)/2,\, s' > 1/2,\, f \in \mathcal{S}(\mathbb{R}^{d+1})$. It holds that 
\EQS{
 &\Bigl\| \int_0^t S(t-t')P_0 f(t')\, dt' \Bigr\|_{L^2_{x_1} L^{\infty}_{x_2, ... , x_{d-1}, t} L^2_{x_d}}  
       \lec \|P_0 f\|_{L^1_{x_1} L^2_{x_2, ... , x_d, t}}, 
                    \label{lowfreqmaximal_2} \\ 
 &\Bigl\| D_{x_1}^{1/2} \int_0^t S(t-t') P_{N,\, N'}f(t')\, dt' \Bigr\|_{L^2_{x_1} L^{\infty}_{x_2, ... , x_{d-1}, t} L^2_{x_d}} 
       \lec N^s \|P_{N,\, N'}f\|_{L^1_{x_1} L^2_{x_2, ... , x_d, t}}. 
                    \label{maximal2}  
} 
The above estimates hold exchanging $x_1$ for $x_i,\, i \in \{ 2, ... , d-1\}$. 
\end{prop}
\begin{proof}
We only check \eqref{maximal2}. 
We assume $|\xi_1| = \max_{1 \le i \le d-1} |\xi_i|$. 
From \eqref{maximal_function_estimate} and the dual estimate of \eqref{local_smoothing_est}
\EQQS{
 \Bigl\| \int_{\mathbb{R}} S(-t') D_{x_1}^{1/2}f(t')\, dt' \Bigr\|_{L^2_x} \lec \|f\|_{L^1_{x_1} L^2_{x_2, ... , x_d, t}},  
}
we obtain for $s > (d-1)/2$ 
\EQQS{
 \Bigl\| D_{x_1}^{1/2} \int_0^T S(t-t') P_{N,\, N'} f(t')\, dt' \Bigr\|_{L^2_{x_1} L^{\infty}_{x_2, ... , x_{d-1}, t} L^2_{x_d}} \lec N^s \|P_{N,\, N'} f\|_{L^1_{x_1} L^2_{x_2, ... , x_d, t}}.  
} 
Then we replace $f(t')$ by $\1_{[0, t]}f(t')$, we obtain the desired result. 
\end{proof} 
We introduce the Strichartz estimate which was proved by \cite{BL} when $d=3$. 
\begin{prop} \label{Strichartz_est} 
 Let $d \ge 2, u_0 \in L^2$. For the admissible pair $(q, r)$, namely $2/q = (d-1)(1/2 - 1/r), (q, r, d) \neq (2, \infty, 3)$, it holds that 
\EQQS{ 
 \|S(t)u_0\|_{L^q_t L^r_{\tilde{x}} L^2_{x_d}} \lec \|u_0\|_2.    
}
\end{prop}
\begin{rem} \label{Strichartz_V^2} 
 For $d \ge 3$, $(q, r) = (4, 2(d-1)/(d-2))$ is an admissible pair, from the property of $U^p, V^p$, 
 it holds that $V^2_S \hookrightarrow U^4_S \hookrightarrow L^4_t L^{2(d-1)/(d-2)}_{\tilde{x}} L^2_{x_d}$. 
 When $d = 2$, $(q, r) = (4, \infty)$ is an admissible pair, then $V^2_S \hookrightarrow U^4_S \hookrightarrow L^4_t L^{\infty}_{x_1} L^2_{x_2}$ holds.  
 See \cite{HHK}. 
\end{rem}
The following trilinear estimates will be applied in the proof of Proposition \ref{BE}. 

\begin{lem} \label{trilinear}
We denote $P_{N_1,\, N_1'} n = n_{N,\, N'},\, u_{N_2,\, N_2'} = u_{N_2,\, N_2'},\, P_{N,\, N'} v = v_{N,\, N'}$ for dyadic numbers $N_1,\, N_2,\, N,\, N_1',\, N_2',\, N'$. Then the following estimates hold:  
$\rm(\hspace{.18em}i\hspace{.18em})$ 
\EQS{
 &\Bigl| \int_{\mathbb{R}^{d+1}} \1_{[0,\, T]} n_{N_1,\, N_1'} u_{N_2,\, N_2'} \overline{v_{N,\, N'}}\, dxdt \Bigr| \notag \\  
 &\lec T^{1/2} N^{-1/2} N_{\min}'^{1/2+} \|n_{N_1,\, N_1'}\|_{L^{\infty}_t L^2_x} \|u_{N_2,\, N_2'}\|_{L^2_{x_1} L^{\infty}_{x_2, ... , x_{d-1}, t} L^2_{x_d}} \|v_{N,\, N'}\|_{U^2_S},  
                          \label{O_11}
}  
$\rm(\hspace{.08em}ii\hspace{.08em})$ If $d \ge 3$, we have 
\EQS{
 &\Bigl| \int_{\mathbb{R}^{d+1}} \1_{[0,\, T]} n_{N_1,\, N_1'} u_{N_2,\, N_2'} \overline{v_{N,\, N'}}\, dxdt \Bigr| \notag \\  
 &\lec T^{1/2} N_1^{(d-3)/2+} N_{\min}'^{1/2+} \|n_{N_1,\, N_1'}\|_{L^{\infty}_t L^2_x} \|u_{N_2,\, N_2'}\|_{V^2_S} \|v_{N,\, N'}\|_{V^2_S},  
             \label{3dim_1}  
}
$\rm(\hspace{.08em}iii\hspace{.08em})$ If $d=2$, we have 
\EQS{
 &\Bigl| \int_{\mathbb{R}^3} \1_{[0,\, T]} n_{N_1,\, N_1'} u_{N_2,\, N_2'} \overline{v_{N,\, N'}}\, dxdt \Bigr|  
 \lec T^{3/4} N_{\min}'^{1/2} \|n_{N_1,\, N_1'}\|_{L^{\infty}_t L^2_x} \|u_{N_2,\, N_2'}\|_{V^2_S} \|v_{N,\, N'}\|_{V^2_S},  
              \label{2d_1} 
}
where $N_{\min}':=\min\{ N',\, N_1',\, N_2'\}$. 
\eqref{O_11} holds exchanging $x_1$ for $x_i,\ i \in \{ 2, ... , d-1 \}$. 
\end{lem}
\begin{proof}
We only show $N_2' = \min\{ N',\, N_1',\, N_2' \}$ since the other cases are treated similarly. 
$\rm(\hspace{.18em}i\hspace{.18em})$ 
Without loss of generality, we assume $|\xi_1| = \max_{1 \le i \le d-1} |\xi_i|$ where $\xi_i$ denotes the $i$-th component of $\xi$. 
By the H\"older inequality and the Sobolev inequality,   
we have 
\EQS{
 (L.H.S.\, of\, \eqref{O_11}) 
 \lec T^{1/2} N_2'^{1/2+} \|n_{N_1,\, N_1'}\|_{L^{\infty}_t L^2_x} \|u_{N_2,\, N_2'}\|_{L^2_{x_1} L^{\infty}_{x_2, ... , x_{d-1}, t} L^2_{x_d}} \|v_{N,\, N'}\|_{L^{\infty}_{x_1} L^2}.  
            \label{estimate_i}
}
Then Remark \ref{loc_smoothing_U^2} leads the desired result. 
$\rm(\hspace{.08em}ii\hspace{.08em})$ follows from the H\"older inequality, the Sobolev inequality and 
Remark \ref{Strichartz_V^2}. Indeed, 
\EQQS{ 
 &(L.H.S.\, of\, \eqref{3dim_1}) \\ 
 &\lec \|n_{N_1,\, N_1'}\|_{L^2_t L^{d-1}_{\tilde{x}} L^2_{x_d}} \|u_{N_2,\, N_2'}\|_{L^4_t L^{2(d-1)/(d-2)}_{\tilde{x}} L^{\infty}_{x_d}} \|v_{N,\, N'}\|_{L^4_t L^{2(d-1)/(d-2)}_{\tilde{x}} L^2_{x_d}} \\ 
 &\lec T^{1/2} N_1^{(d-3)/2+} N_2'^{1/2+} \|n_{N_1,\, N_1'}\|_{L^{\infty}_t L^2_x} \|u_{N_2,\, N_2'}\|_{V^2_S} \|v_{N_2,\, N_2'}\|_{V^2_S}.   
}
By the H\"older inequality, the Sobolev inequality and 
Remark \ref{Strichartz_V^2}, we have 
\EQQS{
 (L.H.S.\, of\, \eqref{2d_1}) 
 &\lec T^{3/4} \|n_{N_1,\, N_1'}\|_{L^{\infty}_t L^2_x} \|u_{N_2,\, N_2'}\|_{L^{\infty}_t L^2_{\tilde{x}} L^{\infty}_{x_d}} \|v_{N,\, N'}\|_{L^4_t L^{\infty}_{\tilde{x}} L^2_{x_d}} \\ 
 &\lec T^{3/4} N_2'^{1/2+} \|n_{N_1,\, N_1'}\|_{L^{\infty}_t L^2_x} \|u_{N_2,\, N_2'}\|_{L^{\infty}_t L^2_x} \|v_{N,\, N'}\|_{V^2_S}. 
}
Then from $V^2_S \hookrightarrow L^{\infty}_t L^2_x$, $\rm(\hspace{.08em}iii\hspace{.08em})$ follows. 
\end{proof}

%%%%%%%%%%%%%%%%%%%%%%%%%%%%%%%%%%%%%%%%%%%%%%%%%%%%%%%%%%%%%%%%%%%%%%%%%%%%%
%%%%%%%%%%%%%%%%%%%%%%%%%%%%%%%%%  Section 3  %%%%%%%%%%%%%%%%%%%%%%%%%%%%%%%
%%%%%%%%%%%%%%%%%%%%%%%%%%%%%%%%%%%%%%%%%%%%%%%%%%%%%%%%%%%%%%%%%%%%%%%%%%%%%
\section{Nonlinear estimates} 
In this section, we derive the nonlinear estimate for the Duhamel term. 
For the linear part, we see from Proposition \ref{local_smoothing} and Proposition \ref{maximal_function} that   
\EQQS{
 &\|P_0 S(t)u_0\|_{V^2_S \, \cap \, L^{\infty}_{x_i} L^2_{x_1, ... , x_{i-1}, x_{i+1}, ... , x_d, t} \, \cap \, L^2_{x_i} L^{\infty}_{x_1, ... , x_{i-1}, x_{i+1}, ... , x_{d-1}, t} L^2_{x_d}} \lec \|P_0 u_0\|_2, \\ 
 &\Bigl( \sum_{N,\, N' \ge 1} N^{2s_k} N'^{2s'} \|P_{N,\, N'} S(t)u_0 \|_{V^2_S}^2  \Bigr)^{\frac{1}{2}} \lec \|u_0\|_{H^{s_k,\, s'}_x}, \\  
 &\Bigl( \sum_{N,\, N' \ge 1} N^{2(s_k + 1/2)} N'^{2s'} \| P_{N,\, N'} S(t)u_0 \|_{L^{\infty}_{x_i} L^2_{x_1, ... , x_{i-1}, x_{i+1}, ... , x_d, t}}^2 \Bigr)^{\frac{1}{2}} 
 \lec \|u_0\|_{H^{s_k,\, s'}_x}, \\ 
 &\Bigl( \sum_{N,\, N' \ge 1} N^{2(s_k - (d-1)/2 - \varepsilon)} N'^{2s'} \|P_{N,\, N'} S(t)u_0 \|_{L^2_{x_i} L^{\infty}_{x_1, ... , x_{i-1}, x_{i+1}, ... , x_{d-1}, t} L^2_{x_d}}^2 \Bigr)^{\frac{1}{2}} 
 \lec \|u_0\|_{H^{s_k,\, s'}_x}, 
}
for $i=1, ... , d-1$. Moreover, 
\EQQS{
 &\|P_0 W_{\pm}(t) n_{\pm 0}\|_{U^2_{W_{\pm}}} 
 + \Bigl( \sum_{N_1,\, N_1' \ge 1} N_1^{2s_l} N_1'^{2s'} \|P_{N_1,\, N_1'} W_{\pm}(t) n_{\pm 0}\|_{U^2_{W_{\pm}}}^2 \Bigr)^{\frac{1}{2}} \\ 
 &\lec \|P_0 n_{\pm 0}\|_2 + \Bigl( \sum_{N_1,\, N_1' \ge 1} N_1^{2s_l} N_1'^{2s'} \|P_{N_1,\, N_1'} n_{\pm 0} \|_2^2 \Bigr)^{\frac{1}{2}}  
 \lec \|n_{\pm 0}\|_{H^{s_l,\, s'}_x}. 
}

Now, we state and prove the nonlinear estimate. 
\begin{prop}  \label{BE}
Let $d \ge 2, s_k > (d-1)/2, s_l > (d-2)/2, s_k - s_l = 1/2, s' > 1/2, u_0 \in H^{s_k,\, s'}, n_{\pm 0} \in H^{s_l,\, s'}, \omega = (-\Delta_{\perp})^{1/2}$ and $0 < T < 1$. 
Then, it holds that  
\EQS{
 &\Bigl\| \int_0^t S(t-t') n(t')u(t')\, dt' \Bigr\|_{X_S}   
      \lec T^{1/2} \|n\|_{Z^{s_l,\, s'}_{W_{\pm}}} \|u\|_{X_S}, 
                          \label{BES} \\
 &\Bigl\| \int_0^t W_{\pm}(t-t') \omega (\bar{u}v)(t')\, dt' \Bigr\|_{Z^{s_l,\, s'}_{W_{\pm}}}   
      \lec T^{1/2} \|u\|_{X_S} \|v\|_{X_S}. 
                          \label{BEW}
} 
\end{prop}
\begin{proof}
%%%%%  B.E. for S  %%%%%% 
To prove \eqref{BES} we estimate the high and the low frequency part of the Duhamel term respectively. 
We firstly estimate the high frequency part of $J$, namely 
\EQQS{ 
 \Bigl( \sum_{N,\, N' \ge 1} N^{2s_k} N'^{2s'} \Bigl\| P_{N,\, N'} \int_0^t S(t-t') n(t')u(t')\, dt' \Bigr\|_{V^2_S}^2 \Bigr)^{1/2}. 
}  
For brevity, we denote $P_{N,\, N'} f = f_{N,\, N'}$.   
We set $J_i,\, J_{4,\, i}, i=1, 2, 3$ where 
\EQQS{
 &J_1 := \Bigl( \sum_{N,\, N' \ge 1} N^{2s_k} N'^{2s'} \sup_{\|v\|_{U^2_S}=1} \Bigl| \sum_{N_2 \lec N} \sum_{N_1 \sim N} \int_{\mathbb{R}^{d+1}} \1_{[0,\, T]} n_{N_1} u_{N_2} \overline{v_{N,\, N'}}\, dxdt \Bigr|^2 \Bigr)^{1/2}, \\ 
 &J_2 := \Bigl( \sum_{N,\, N' \ge 1} N^{2s_k} N'^{2s'} \sup_{\|v\|_{U^2_S}=1} \Bigl| \sum_{N_2 \gg N} \sum_{N_1 \sim N_2} \int_{\mathbb{R}^{d+1}} \1_{[0,\, T]} n_{N_1} u_{N_2} \overline{v_{N,\, N'}}\, dxdt \Bigr|^2 \Bigr)^{1/2}, \\ 
 &J_3 := \Bigl( \sum_{N,\, N' \ge 1} N^{2s_k} N'^{2s'} \sup_{\|v\|_{U^2_S}=1} \Bigl| \sum_{N_2 \sim N} \sum_{N_1 \ll N_2} \int_{\mathbb{R}^{d+1}} \1_{[0,\, T]} n_{N_1} u_{N_2} \overline{v_{N,\, N'}}\, dxdt \Bigr|^2 \Bigr)^{1/2}, \\ 
 &J_{4,\, 1} := \Bigl( \sum_{N,\, N' \ge 1} N^{2s_k} N'^{2s'} \Bigl\| P_{N,\, N'} \int_0^t S(t-t') \sum_{N_2 \lec N} \sum_{N_1 \sim N} n_{N_1} u_{N_2}\, dt' \Bigr\|_{L^{\infty}_t L^2_x}^2 \Bigr)^{1/2}, \\ 
 &J_{4,\, 2} := \Bigl( \sum_{N,\, N' \ge 1} N^{2s_k} N'^{2s'} \Bigl\| P_{N,\, N'} \int_0^t S(t-t') \sum_{N_2 \gg N} \sum_{N_1 \sim N_2} n_{N_1} u_{N_2}\, dt' \Bigr\|_{L^{\infty}_t L^2_x}^2 \Bigr)^{1/2}, \\ 
 &J_{4,\, 3} := \Bigl( \sum_{N,\, N' \ge 1} N^{2s_k} N'^{2s'} \Bigl\| P_{N,\, N'} \int_0^t S(t-t') \sum_{N_2 \sim N} \sum_{N_1 \ll N_2} n_{N_1} u_{N_2}\, dt' \Bigr\|_{L^{\infty}_t L^2_x}^2 \Bigr)^{1/2}. 
}
Then from \cite{KaT} Corollary 2.10 
\EQQS{
 \Bigl( \sum_{N,\, N' \ge 1} N^{2s_k} N'^{2s'} \Bigl\| P_{N,\, N'} \int_0^t S(t-t') n(t')u(t')\, dt' \Bigr\|_{V^2_S}^2 \Bigr)^{1/2} \lec \sum_{i=1}^3 (J_i + J_{4,\, i}).  
}
Let us separate $J_1$ into the following three parts. 
\EQQS{
 J_{1,\, 1} &:= \Bigl( \sum_{N,\, N' \ge 1} N^{2s_k} N'^{2s'} \sup_{\|v\|_{U^2_S}=1} \Bigl| \sum_{N_2 \lec N,\, N_2' \lec N'} \sum_{N_1 \sim N,\, N_1' \sim N'} \\  
 &\hspace{15em} \int_{\mathbb{R}^{d+1}} \1_{[0,\, T]} n_{N_1,\, N_1'} u_{N_2,\, N_2'} \overline{v_{N,\, N'}}\, dxdt \Bigr|^2 \Bigr)^{1/2}, \\ 
 J_{1,\, 2} &:= \Bigl( \sum_{N,\, N' \ge 1} N^{2s_k} N'^{2s'} \sup_{\|v\|_{U^2_S}=1} \Bigl| \sum_{N_2 \lec N,\, N_2' \gg N'} \sum_{N_1 \sim N,\, N_1' \sim N_2'} \\  
 &\hspace{15em} \int_{\mathbb{R}^{d+1}} \1_{[0,\, T]} n_{N_1,\, N_1'} u_{N_2,\, N_2'} \overline{v_{N,\, N'}}\, dxdt \Bigr|^2 \Bigr)^{1/2}, \\ 
 J_{1,\, 3} &:= \Bigl( \sum_{N,\, N' \ge 1} N^{2s_k} N'^{2s'} \sup_{\|v\|_{U^2_S}=1} \Bigl| \sum_{N_2 \lec N,\, N_2' \sim N'} \sum_{N_1 \sim N,\, N_1' \ll N_2'} \\ 
 &\hspace{15em} \int_{\mathbb{R}^{d+1}} \1_{[0,\, T]} n_{N_1,\, N_1'} u_{N_2,\, N_2'} \overline{v_{N,\, N'}}\, dxdt \Bigr|^2 \Bigr)^{1/2}. 
} 
Similarly, we have $J_i \lec \sum_{l=1}^3 J_{i,\, l},\, J_{4,\, k} \lec \sum_{l=1}^3 J_{4,\, k,\, l}, (i=2, 3, k=1, 2, 3)$.  
For brevity, we only consider the case $N_2' \lec N' \sim N_1$, namely $J_{1,\, 1}$  since the other cases are treated similarly from $L^{\infty}(\mathbb{R}) \hookrightarrow H^{s'}(\mathbb{R})$ with $s' > 1/2$. 
Hereafter we consider 
$|\xi_1| = \max_{1 \le i \le d-1} |\xi_i|$ for brevity. 
Now we estimate $J_{1,\, 1}$. 
From Lemma \ref{trilinear} \eqref{O_11}, $s_k - s_l = 1/2, s_k - (d-1)/2 - \varepsilon > 0$, the Cauchy-Schwarz inequality and $\|n\|_{L^{\infty}_t H^{s_l,\, s'}_x} \lec R$, we have   
\EQQS{ 
 J_{1,\, 1} &\lec T^{1/2} \Bigl( \sum_{N,\, N' \ge 1} N^{2s_k-1} N'^{2s'} \Bigl( \sum_{N_1 \sim N,\, N_1' \sim N'} \|n_{N_1,\, N_1'}\|_{L^{\infty}_t L^2_x} \Bigl( \|P_0 u\|_{L^2_{x_1} L^{\infty}_{x_2, ... , x_{d-1}, t} L^2_{x_d}} \\ 
 &\qquad + \sum_{1 \le N_2 \lec N,\, 1 \le N_2' \lec N'} N_2'^{1/2+} \|u_{N_2,\, N_2'}\|_{L^2_{x_1} L^{\infty}_{x_2, ... , x_{d-1}, t} L^2_{x_d}} \Bigr)  \Bigr)^2 \Bigr)^{1/2} \\  
 &\lec T^{1/2} \Bigl( \sum_{N_1,\, N_1' \gec 1} N_1^{2s_l} N_1'^{2s'} \|n_{N_1,\, N_1'}\|_{L^{\infty}_t L^2_x}^2 \Bigl( \|P_0 u\|_{L^2_{x_1} L^{\infty}_{x_2, ... , x_{d-1}, t} L^2_{x_d}}^2 \\ 
 &\qquad + \sum_{1 \le N_2 \lec N_1,\, 1 \le N_2' \lec N_1'} N_2^{2(s_k-(d-1)/2-\varepsilon)} N_2'^{2s'} \|u_{N_2,\, N_2'}\|_{L^2_{x_1} L^{\infty}_{x_2, ... , x_{d-1}, t} L^2_{x_d}}^2 \Bigr)      
  \Bigr)^{1/2} \\ 
 &\lec T^{1/2} M \|n\|_{L^{\infty}_t H^{s_l,\, s'}_x}  
 \lec T^{1/2} MR.   
}  
Let us estimate $J_{2,\, 1}$. 
From Lemma \ref{trilinear} \eqref{O_11}, $s_k - s_l = 1/2$, the Cauchy-Schwarz inequality, $l^1 l^2 \hookrightarrow l^2 l^1$ and $s_k - (d-1)/2 - \varepsilon > 0$, we have 
\EQQS{
 J_{2,\, 1} &= \Bigl( \sum_{N,\, N' \ge 1} N^{2s_k} N'^{2s'} \sup_{\|v\|_{U^2_S}=1} \Bigl| \sum_{N_2 \gg N,\, N_2' \lec N'} \sum_{N_1 \sim N_2,\, N_1' \sim N'} \\
 &\qquad \int_{\mathbb{R}^{d+1}} \1_{[0,\, T]} n_{N_1,\, N_1'} u_{N_2,\, N_2'} \overline{v_{N,\, N'}}\, dxdt \Bigr|^2 \Bigr)^{1/2} \\  
 &\lec T^{1/2} \Bigl( \sum_{N,\, N' \ge 1} N^{2s_k-1} N'^{2s'} \Bigl( \sum_{N_2 \gg N,\, N_2' \lec N'} \sum_{N_1 \sim N_2,\, N_1' \sim N'} N_2'^{1/2+} \|n_{N_1,\, N_1'}\|_{L^{\infty}_t L^2_x} \\ 
 &\qquad \|u_{N_2,\, N_2'}\|_{L^2_{x_1} L^{\infty}_{x_2, ... , x_{d-1}, t} L^2_{x_d}}  \Bigr)^2 \Bigr)^{1/2} \\  
 &\lec T^{1/2} \sum_{N_2 \gg 1} \sum_{N_1 \sim N_2} \Bigl( N_1^{2s_l} \|n_{N_1}\|_{L^{\infty}_t L^2_{\tilde{x}} H^{s'}_{x_d}}^2 \|u_{N_2}\|_{L^2_{x_1} L^{\infty}_{x_2, ... , x_{d-1}, t} H^{s'}_{x_d}}^2 \Bigr)^{1/2} \\ 
 &\lec T^{1/2} \Bigl( \sum_{N_2 \gg 1} \|u_{N_2}\|_{L^2_{x_1} L^{\infty}_{x_2, ... , x_{d-1}, t} H^{s'}_{x_d}}^2 \Bigr)^{1/2} \|n\|_{L^{\infty}_t H^{s_l,\, s'}_x}  
 \lec T^{1/2} MR. 
}  
Let us estimate $J_{3,\, 1}$. 
When $d \ge 3$, from Lemma \ref{trilinear} \eqref{3dim_1}, $U^2_S \hookrightarrow V^2_S$, $s_l > (d-2)/2$ and the Cauchy-Schwarz inequality we have 
\EQQS{
 J_{3,\, 1} 
 &= \Bigl( \sum_{N,\, N' \ge 1} N^{2s_k} N'^{2s'} \sup_{\|v\|_{U^2_S}=1} \Bigl| \sum_{N_2 \sim N,\, N_2' \lec N'} \sum_{N_1 \ll N_2,\, N_1' \sim N'} \\ 
 &\qquad \int_{\mathbb{R}^{d+1}} \1_{[0,\, T]} n_{N_1,\, N_1'} u_{N_2,\, N_2'} \overline{v_{N,\, N'}}\, dxdt \Bigr|^2 \Bigr)^{1/2} \\
 &\lec T^{1/2} \Bigl( \sum_{N,\, N' \ge 1} N^{2s_k} N'^{2s'} \Bigl( \sum_{N_2 \sim N,\, N_2' \lec N'} \sum_{N_1 \ll N_2,\, N_1' \sim N'} N_1^{(d-3)/2+} N_2'^{1/2+} \\ 
 &\qquad \|n_{N_1,\, N_1'}\|_{L^{\infty}_t L^2_x}  \|u_{N_2,\, N_2'}\|_{V^2_S} \Bigr)^2 \Bigr)^{1/2} \\  
 &\lec T^{1/2} \|u\|_{Y^{s_k,\, s'}_S} \|n\|_{L^{\infty}_t H^{s_l,\, s'}_x} 
 \lec T^{1/2} JR.   
}
For $d = 2$, by Lemma \ref{trilinear} \eqref{2d_1} we have 
\EQQS{
 J_{3,\, 1} &\lec T^{3/4} \Bigl( \sum_{N,\, N' \ge 1} N^{2s_k} N'^{2s'} \Bigl( \sum_{N_2 \sim N,\, N_2' \lec N'} \sum_{N_1 \ll N_2,\, N_1' \sim N'} N_2'^{1/2+} \\ 
 &\qquad \|n_{N_1,\, N_1'}\|_{L^{\infty}_t L^2_x} \|u_{N_2,\, N_2'}\|_{V^2_S} \Bigr)^2 \Bigr)^{1/2} \\ 
 &\lec T^{3/4} \|u\|_{Y^{s_k,\, s'}_S} \|n\|_{L^{\infty}_t H^{s_l,\, s'}_x} 
 \lec T^{3/4} JR. 
} 
We estimate $J_{4,\, 1,\, 1}$ below. 
From $|\xi_1| = \max_{1 \le i \le d-1} |\xi_i|$, Proposition \ref{inhomo_L^infty_tL^2_x} \eqref{conti}, $s_k - s_l = 1/2$, the H\"older inequality, $s_k - (d-1)/2 - \varepsilon > 0$ and the Cauchy-Schwarz inequality  
\EQS{
 J_{4,\, 1,\, 1} 
 &= \Bigl( \sum_{N,\, N' \ge 1} N^{2s_k} N'^{2s'} \Bigl\| P_{N,\, N'} \int_0^t S(t-t') \sum_{N_2 \lec N,\, N_2' \lec N'} \sum_{N_1 \sim N,\, N_1' \sim N'} \notag \\ 
 &\qquad n_{N_1,\, N_1'} u_{N_2,\, N_2'}\, dt' \Bigr\|_{L^{\infty}_t L^2_x}^2 \Bigr)^{1/2} \notag \\ 
 &\lec \Bigl( \sum_{N,\, N' \ge 1} N^{2s_k-1} N'^{2s'} \Bigl\| P_{N,\, N'} \Bigl( \sum_{N_2 \lec N,\, N_2' \lec N'} \sum_{N_1 \sim N,\, N_1' \sim N'} \notag \\ 
 &\qquad n_{N_1,\, N_1'} u_{N_2,\, N_2'}\Bigr) \Bigr\|_{L^1_{x_1} L^2_{x_2, ... , x_d, t}}^2 \Bigr)^{1/2} 
                    \label{O_411} \\ 
 &\lec \Bigl( \sum_{N_1,\, N_1' \gec 1} T N_1^{2s_l} N_1'^{2s'} \Bigl\| \sum_{N_2 \lec N_1,\, N_2' \lec N_1'} u_{N_2,\, N_2'} \Bigr\|_{L^2_{x_1} L^{\infty}_{x_2, ... , x_d, t}}^2 \|n_{N_1,\, N_1'}\|_{L^{\infty}_t L^2_x}^2 \Bigr)^{1/2} \notag \\     
 &\lec T^{1/2} M \|n\|_{L^{\infty}_t H^{s_l,\, s'}_x} 
 \lec T^{1/2} MR. \notag 
}   
We estimate $J_{4,\, 2,\, 1}$. 
From $|\xi_1| = \max_{1 \le i \le d-1} |\xi_i|$, Proposition \ref{inhomo_L^infty_tL^2_x} \eqref{conti}, $s_k - s_l = 1/2$, the H\"older inequality and the Sobolev inequality   
\EQS{
 J_{4,\, 2,\, 1} 
 &= \Bigl( \sum_{N,\, N' \ge 1} N^{2s_k} N'^{2s'} \Bigl\| P_{N,\, N'} \int_0^t S(t-t') \sum_{N_2 \gg N,\, N_2' \lec N'} \sum_{N_1 \sim N_2,\, N_1' \sim N'} \notag \\
 &\qquad n_{N_1,\, N_1'} u_{N_2,\, N_2'}\, dt' \Bigr\|_{L^{\infty}_t L^2_x}^2 \Bigr)^{1/2} \notag \\ 
 &\lec \Bigl( \sum_{N,\, N' \ge 1} N^{2s_k-1} N'^{2s'} \Bigl\| \sum_{N_2 \gg N,\, N_2' \lec N'} \sum_{N_1 \sim N_2,\, N_1' \sim N'} n_{N_1,\, N_1'} u_{N_2,\, N_2'} \Bigr\|_{L^1_{x_1} L^2}^2 \Bigr)^{1/2} 
                                 \label{O_421} \\ 
 &\lec T^{1/2} \Bigl( \sum_{N_1 \gg 1,\, N_1' \gec 1} N_1^{2s_l} N_1'^{2s'} \Bigl\| \sum_{N_2 \sim N_1,\, N_2' \lec N_1'} u_{N_2,\, N_2'} \Bigr\|_{L^2_{x_1} L^{\infty}}^2 \|n_{N_1,\, N_1'}\|_{L^{\infty}_t L^2_x}^2 \Bigr)^{1/2} \notag \\ 
 &\lec T^{1/2} M \|n\|_{L^{\infty}_t H^{s_l,\, s'}_x}  
 \lec T^{1/2} MR.     \notag 
}   
Now we estimate $J_{4,\, 3,\, 1}$. 
By $U^2_S \hookrightarrow L^{\infty}_t L^2_x$, we have 
\EQS{
 J_{4,\, 3,\, 1} 
 &= \Bigl( \sum_{N,\, N' \ge 1} N^{2s_k} N'^{2s'} \Bigl\| P_{N,\, N'} \int_0^t S(t-t') \sum_{N_2 \sim N,\, N_2' \lec N'} \sum_{N_1 \ll N_2,\, N_1' \sim N'} \notag \\
 &\qquad n_{N_1,\, N_1'} u_{N_2,\, N_2'}\, dt' \Bigr\|_{L^{\infty}_t L^2_x}^2 \Bigr)^{1/2} \notag \\ 
 &\lec \Bigl( \sum_{N,\, N' \ge 1} N^{2s_k} N'^{2s'} \Bigl\| P_{N,\, N'} \int_0^t S(t-t') \sum_{N_2 \sim N,\, N_2' \lec N'} \sum_{N_1 \ll N_2,\, N_1' \sim N'} \notag \\ 
 &\qquad n_{N_1,\, N_1'} u_{N_2,\, N_2'}\, dt' \Bigr\|_{U^2_S}^2 \Bigr)^{1/2} \notag \\ 
 &\lec \Bigl( \sum_{N,\, N' \ge 1} N^{2s_k} N'^{2s'} \sup_{\|v\|_{V^2_S}=1} \Bigl|  \sum_{N_2 \sim N,\, N_2' \lec N'} \sum_{N_1 \ll N_2,\, N_1' \sim N'} \notag \\ 
 &\qquad \int_{\mathbb{R}^{d+1}} \1_{[0,\, T]} n_{N_1,\, N_1'} u_{N_2,\, N_2'} \overline{v_{N,\, N'}} \, dxdt \Bigr|^2 \Bigr)^{1/2}. 
                   \label{O_431} 
} 
For $d \ge 3$, we apply Lemma \ref{trilinear} \eqref{3dim_1}, $s_l > (d-2)/2$ and the Cauchy-Schwarz inequality, then we have   
\EQQS{ 
 (R.H.S.\, of\, \eqref{O_431}) 
 &\lec \Bigl( \sum_{N,\, N' \ge 1} N^{2s_k} N'^{2s'} \Bigl( \sum_{N_2 \sim N,\, N_2' \lec N'} \sum_{N_1 \ll N_2,\, N_1' \sim N'} T^{1/2} \\   
 &\qquad N_1^{(d-3)/2+} N_2'^{1/2+} \|n_{N_1,\, N_1'}\|_{L^{\infty}_t L^2_x} \|u_{N_2,\, N_2'}\|_{V^2_S} \Bigr)^2 \Bigr)^{1/2} \\  
 &\lec T^{1/2} \|u\|_{Y^{s_k,\, s'}_S} \|n\|_{L^{\infty}_t H^{s_l,\, s'}_x}  
 \lec T^{1/2} JR.   
}
For $d=2$, by Lemma \ref{trilinear} \eqref{2d_1} and the Cauchy-Schwarz inequality,   
\EQQS{ 
 (R.H.S.\, of\, \eqref{O_431}) 
 &\lec \Bigl( \sum_{N,\, N' \ge 1} N^{2s_k} N'^{2s'} \Bigl( \sum_{N_2 \sim N,\, N_2' \lec N'} \sum_{N_1 \ll N_2,\, N_1' \sim N'} T^{3/4} \\ 
 &\qquad N_2'^{1/2+} \|n_{N_1,\, N_1'}\|_{L^{\infty}_t L^2_x} \|u_{N_2,\, N_2'}\|_{V^2_S} \Bigr)^2 \Bigr)^{1/2} \\ 
 &\lec T^{3/4} \|u\|_{Y^{s_k,\, s'}_S} \|n\|_{L^{\infty}_t H^{s_l,\, s'}_x}  
 \lec T^{3/4}JR. 
}
Let us estimate $K$. 
We first estimate high frequency part of the Duhamel term  
and we only consider the case $|\xi_1| = \max_{1 \le i \le d-1} |\xi_i|$, namely   
\EQQS{
 \Bigl( \sum_{N,\, N' \ge 1} N^{2(s_k + 1/2)} N'^{2s'} \Bigl\|P_{N,\, N'} \int_0^t S(t-t') (nu)(t') dt'\Bigr\|_{L^{\infty}_{x_1} L^2_{x_2, ... , x_d, t}}^2 \Bigr)^{1/2}. 
}
The above term is bounded by  
\EQQS{
 &\Bigl( \sum_{N,\, N' \ge 1} N^{2(s_k+1/2)} N'^{2s'} \Bigl\| P_{N,\, N'} \int_0^t S(t-t') \Bigl( \sum_{N_2 \lec N} \sum_{N_1 \sim N} + \! \sum_{N_2 \gg N} \sum_{N_1 \sim N_2} + \! \sum_{N_2 \sim N} \sum_{N_1 \ll N_2} \Bigr) \\ 
 &\qquad (n_{N_1} u_{N_2})(t')\, dt' \Bigr\|_{L^{\infty}_{x_1} L^2_{x_2, ... , x_d, t}}^2 \Bigr)^{1/2} \\ 
 &\lec K_1 + K_2 + K_3,  
} 
where 
\EQQS{
 K_1^2 &:= \sum_{N,\, N' \ge 1} N^{2(s_k+1/2)} N'^{2s'} \Bigl\| P_{N,\, N'} \int_0^t S(t-t') \sum_{N_2 \lec N} \sum_{N_1 \sim N} (n_{N_1} u_{N_2})(t')\, dt' \Bigr\|_{L^{\infty}_{x_1} L^2}^2, \\
 K_2^2 &:= \sum_{N,\, N' \ge 1} N^{2(s_k+1/2)} N'^{2s'} \Bigl\| P_{N,\, N'} \int_0^t S(t-t') \sum_{N_2 \gg N} \sum_{N_1 \sim N_2} (n_{N_1} u_{N_2})(t')\, dt' \Bigr\|_{L^{\infty}_{x_1} L^2}^2, \\
 K_3^2 &:= \sum_{N,\, N' \ge 1} N^{2(s_k+1/2)} N'^{2s'} \Bigl\| P_{N,\, N'} \int_0^t S(t-t') \sum_{N_2 \sim N} \sum_{N_1 \ll N_2} (n_{N_1} u_{N_2})(t')\, dt' \Bigr\|_{L^{\infty}_{x_1} L^2}^2.  
}
Let us separate $K_1$ into the following three parts. 
\EQQS{
 K_{1,\, 1} &:= \Bigl( \sum_{N,\, N' \ge 1} N^{2(s_k+1/2)} N'^{2s'} \Bigl\| P_{N,\, N'} \int_0^t S(t-t') \sum_{N_2 \lec N,\, N_2' \lec N'} \sum_{N_1 \sim N,\, N_1' \sim N'} \\ 
 &\qquad n_{N_1,\, N_1'} u_{N_2,\, N_2'}\, dt' \Bigr\|_{L^{\infty}_{x_1} L^2_{x_2, ... , x_d, t}}^2 \Bigr)^{1/2}, \\
 K_{1,\, 2} &:= \Bigl( \sum_{N,\, N' \ge 1} N^{2(s_k+1/2)} N'^{2s'} \Bigl\| P_{N,\, N'} \int_0^t S(t-t') \sum_{N_2 \lec N,\, N_2' \gg N'} \sum_{N_1 \sim N,\, N_1' \sim N_2'} \\ 
 &\qquad n_{N_1,\, N_1'} u_{N_2,\, N_2'}\, dt' \Bigr\|_{L^{\infty}_{x_1} L^2_{x_2, ... , x_d, t}}^2 \Bigr)^{1/2}, \\
 K_{1,\, 3} &:= \Bigl( \sum_{N,\, N' \ge 1} N^{2(s_k+1/2)} N'^{2s'} \Bigl\| P_{N,\, N'} \int_0^t S(t-t') \sum_{N_2 \lec N,\, N_2' \sim N'} \sum_{N_1 \sim N,\, N_1' \ll N_2'} \\ 
 &\qquad n_{N_1,\, N_1'} u_{N_2,\, N_2'}\, dt' \Bigr\|_{L^{\infty}_{x_1} L^2_{x_2, ... , x_d, t}}^2 \Bigr)^{1/2}. 
} 
Likewise, we have $K_i \lec \sum_{j=1}^3 K_{i,\, j},\, i=2, 3$. 
For simplicity, we only show the case $N_2' \lec N' \sim N_1'$, namely $K_{i,\, 1},\, i=1, 2, 3$. 
We estimate $K_{1,\, 1}$ as follows.  
By $|\xi_1| = \max_{1 \le i \le d-1} |\xi_i|$, Proposition \ref{inhomo_local_smoothing} \eqref{smoothing}, \eqref{O_411} and the estimate for $J_{4,\, 1,\, 1}$,  
\EQQS{
 K_{1,\, 1} &\lec \Bigl( \sum_{N,\, N' \ge 1} N^{2(s_k+1/2)-2} N'^{2s'} \Bigl\| P_{N,\, N'} \Bigl( \sum_{N_2 \lec N,\, N_2' \lec N'} \sum_{N_1 \sim N,\, N_1' \sim N'} \\ 
 &\qquad n_{N_1,\, N_1'} u_{N_2,\, N_2'} \Bigr) \Bigr\|_{L^1_{x_1} L^2_{x_2, ... , x_d, t}}^2 \Bigr)^{1/2} \\ 
 &\lec T^{1/2} MR.  
} 
We estimate $K_{2,\, 1}$. 
By Proposition \ref{inhomo_local_smoothing} \eqref{smoothing}, \eqref{O_421} and the estimate for $J_{4,\, 2,\, 1}$, 
\EQQS{
 K_{2,\, 1} 
 &:= \Bigl( \sum_{N,\, N' \ge 1} N^{2(s_k+1/2)} N'^{2s'} \Bigl\| P_{N,\, N'} \int_0^t S(t-t') \sum_{N_2 \gg N,\, N_2' \lec N'} \sum_{N_1 \sim N_2,\, N_1' \sim N'} \\ 
 &\qquad n_{N_1,\, N_1'} u_{N_2,\, N_2'}\, dt' \Bigr\|_{L^{\infty}_{x_1} L^2_{x_2, ... , x_d, t}}^2 \Bigr)^{1/2} \\
 &\lec \Bigl(\sum_{N,\, N' \ge 1} N^{2(s_k+1/2)-2} N'^{2s'} \Bigl\| P_{N,\, N'} \sum_{N_2 \gg N,\, N_2' \lec N'} \sum_{N_1 \sim N_2,\, N_1' \sim N'} \\ 
 &\qquad n_{N_1,\, N_1'} u_{N_2,\, N_2'} \Bigr\|_{L^1_{x_1} L^2_{x_2, ... , x_d, t}}^2 \Bigr)^{1/2} \\  
 &\lec T^{1/2} MR.    
}                                               
We estimate $K_{3,\, 1}$. 
From $|\xi_1| = \max_{1 \le i \le d-1} |\xi_i|$ and Remark \ref{loc_smoothing_U^2}, we have 
\EQQS{
 K_{3,\, 1} 
 &:= \Bigl( \sum_{N,\, N' \ge 1} N^{2(s_k+1/2)} N'^{2s'} \Bigl\| P_{N,\, N'} \int_0^t S(t-t') \sum_{N_2 \sim N,\, N_2' \lec N'} \sum_{N_1 \ll N_2,\, N_1' \sim N'} \\ 
 &\qquad n_{N_1,\, N_1'} u_{N_2,\, N_2'}\, dt' \Bigr\|_{L^{\infty}_{x_1} L^2}^2 \Bigr)^{1/2} \\ 
 &\lec \Bigl( \sum_{N,\, N' \ge 1} N^{2(s_k+1/2)-1} N'^{2s'} \Bigl\| P_{N,\, N'} \int_0^t S(t-t') \sum_{N_2 \sim N,\, N_2' \lec N'} \sum_{N_1 \ll N_2,\, N_1' \sim N'} \\
 &\qquad n_{N_1,\, N_1'} u_{N_2,\, N_2'}\, dt' \Bigr\|_{U^2_S}^2 \Bigr)^{1/2} \\ 
 &\lec \Bigl( \sum_{N,\, N' \ge 1} N^{2s_k} N'^{2s'} \sup_{\|v\|_{V^2_S}=1} \Bigl| \sum_{N_2 \sim N,\, N_2' \lec N'} \sum_{N_1 \ll N_2,\, N_1' \sim N'}  \\ 
 &\qquad \int_{\mathbb{R}^{d+1}} \1_{[0,\, T]} n_{N_1,\, N_1'} u_{N_2,\, N_2'} \overline{v_{N,\, N'}}\, dxdt \Bigr|^2 \Bigr)^{1/2}.  
              %%    \label{higher_dim_1}
}
From \eqref{O_431} and the estimate for $J_{4,\, 3,\, 1}$, we obtain 
\EQQS{ 
 K_{3,\, 1} \lec \begin{cases} 
                   T^{1/2} JR \qquad (d \ge 3), \\ 
                   T^{3/4} JR \qquad (d=2). 
                 \end{cases}
}

Let us estimate $M$. 
We only estimate the following high frequency part of the Duhamel term, namely    
\EQQS{
 \Bigl( \sum_{N,\, N' \ge 1} N^{2(s_k-(d-1)/2-\varepsilon)} N'^{2s'} \Bigl\| P_{N,\, N'}\int_0^t S(t-t') (nu)(t') \, dt' \Bigr\|_{L^2_{x_1} L^{\infty}_{x_2, ... , x_{d-1}, t} L^2_{x_d}}^2 \Bigr)^{1/2}.  
}
The above term is bounded by 
\EQQS{
 &\Bigl( \sum_{N,\, N' \ge 1} N^{2(s_k-(d-1)/2-\varepsilon)} N'^{2s'} \Bigl\| P_{N,\, N'}\int_0^t S(t-t') \Bigl( \sum_{N_2 \lec N} \sum_{N_1 \sim N} + \sum_{N_2 \gg N} \sum_{N_1 \sim N_2} \\ 
 &\qquad + \sum_{N_2 \sim N} \sum_{N_1 \ll N_2} \Bigr) (n_{N_1}u_{N_2})(t') \, dt' \Bigr\|_{L^2_{x_1} L^{\infty}_{x_2, ... , x_{d-1}, t} L^2_{x_d}}^2 \Bigr)^{1/2} \\ 
 &\lec M_1 + M_2 + M_3,   
}
where 
\EQQS{
 M_1 &:= \Bigl( \sum_{N,\, N' \ge 1} N^{2(s_k-(d-1)/2-\varepsilon)} N'^{2s'} \Bigl\| P_{N,\, N'}\int_0^t S(t-t') \\ 
 &\qquad \sum_{N_2 \lec N} \sum_{N_1 \sim N} (n_{N_1}u_{N_2})(t') \, dt' \Bigr\|_{L^2_{x_1} L^{\infty}_{x_2, ... , x_{d-1}, t} L^2_{x_d}}^2 \Bigr)^{1/2}, \\ 
 M_2 &:= \Bigl( \sum_{N,\, N' \ge 1} N^{2(s_k-(d-1)/2-\varepsilon)} N'^{2s'} \Bigl\| P_{N,\, N'}\int_0^t S(t-t') \\ 
 &\qquad \sum_{N_2 \gg N} \sum_{N_1 \sim N_2} (n_{N_1}u_{N_2})(t') \, dt' \Bigr\|_{L^2_{x_1} L^{\infty}_{x_2, ... , x_{d-1}, t} L^2_{x_d}}^2 \Bigr)^{1/2}, \\ 
 M_3 &:= \Bigl( \sum_{N,\, N' \ge 1} N^{2(s_k-(d-1)/2-\varepsilon)} N'^{2s'} \Bigl\| P_{N,\, N'}\int_0^t S(t-t') \\ 
 &\qquad \sum_{N_2 \sim N} \sum_{N_1 \ll N_2} (n_{N_1}u_{N_2})(t') \, dt' \Bigr\|_{L^2_{x_1} L^{\infty}_{x_2, ... , x_{d-1}, t} L^2_{x_d}}^2 \Bigr)^{1/2}.    
}
From Proposition \ref{inhomo_maximal_fcn} \eqref{maximal2} and $|\xi_1| = \max_{1 \le i \le d-1} |\xi_i|$,  $M_1$ is bounded by 
\EQS{
 \Bigl( \sum_{N,\, N' \ge 1} N^{2(s_k - 1/2)} N'^{2s'} \| P_{N,\, N'} \sum_{N_2 \lec N} \sum_{N_1 \sim N} n_{N_1}u_{N_2}\|_{L^1_{x_1} L^2_{x_2, ... , x_d, t}} \Bigr)^{1/2}. 
                        \label{est_J} 
}
Then by $s_k - s_l = 1/2$, \eqref{est_J} is bounded by $M_{1,\, i},\, i=1, 2, 3$,   
\EQQS{
 M_1 &\lec \Bigl( \sum_{N,\, N' \ge 1} N^{2s_l} N'^{2s'} \Bigl\| \Bigl( \sum_{N_2 \lec  N,\,  N_2' \lec N'}\sum_{N_1 \sim N,\, N_1' \sim N'} + \sum_{N_2 \lec N,\,  N_2' \gg N'}\sum_{N_1 \sim N,\, N_1' \sim N_2'} \\ 
 &\qquad + \sum_{N_2 \lec N,\, N_2' \sim N'}\sum_{N_1 \sim N,\,  N_1' \ll N_2'} \Bigr) P_{N,\, N'}(n_{N_1, N_1'}\, u_{N_2, N_2'}) \Bigr\|_{L^1_{x_1} L^2_{x_2, ... , x_d, t}}^2 \Bigr)^{1/2}, \\  
 &\lec M_{1,\, 1} + M_{1,\, 2} + M_{1,\, 3},  
}
where 
\EQQS{
 M_{1,\, 1}^2 &:= \sum_{N,\, N' \ge 1} N^{2s_l} N'^{2s'} \Bigl\| P_{N,\, N'} \Bigl( \sum_{N_2 \lec  N,\,  N_2' \lec N'}\sum_{N_1 \sim N,\, N_1' \sim N'} n_{N_1, N_1'}\, u_{N_2, N_2'} \Bigr) \Bigr\|_{L^1_{x_1} L^2}^2, \\ 
 M_{1,\, 2}^2 &:= \sum_{N,\, N' \ge 1} N^{2s_l} N'^{2s'} \Bigl\| P_{N,\, N'} \Bigl( \sum_{N_2 \lec N,\,  N_2' \gg N'}\sum_{N_1 \sim N,\, N_1' \sim N_2'} n_{N_1, N_1'}\, u_{N_2, N_2'} \Bigr) \Bigr\|_{L^1_{x_1} L^2}^2, \\ 
 M_{1,\, 3}^2 &:= \sum_{N,\, N' \ge 1} N^{2s_l} N'^{2s'} \Bigl\| P_{N,\, N'} \Bigl( \sum_{N_2 \lec N,\, N_2' \sim N'}\sum_{N_1 \sim N,\,  N_1' \ll N_2'} n_{N_1, N_1'}\, u_{N_2, N_2'} \Bigr) \Bigr\|_{L^1_{x_1} L^2}^2. 
}
From \eqref{O_411} and the estimate for $J_{4,\, 1,\, 1}$, we obtain 
$M_{1,\, 1} \lec T^{1/2}MR$.  
Similarly we can check $M_{1,\, i} \lec T^{1/2} MR, i = 2, 3$. 
For the estimate of $M_2$, we separate $M_2$ into three parts $M_{2,\, i}, i=1, 2, 3$ as $M_{1,\, i}$ above, then we can obtain $M_{2,\, i} \lec T^{1/2}MR, i=1, 2, 3$.  
For instance, the estimate for $M_{2,\, 1}$ is derived by \eqref{O_421} and the estimate for $J_{4,\, 2,\, 1}$.  
Now we estimate $M_3$. 
From Proposition \ref{maximal_function} \eqref{maximal_function_estimate}, 
$M_3$ is bounded by 
\EQQS{
 &\Bigl( \sum_{N,\, N' \ge 1} N^{2s_k} N'^{2s'} \Bigl\| P_{N,\, N'} \int_0^t S(t-t') \sum_{N_2 \sim N} \sum_{N_1 \ll N_2} (n_{N_1} u_{N_2})(t')\, dt' \Bigr\|_{U^2_S}^2 \Bigr)^{1/2} \\ 
 &\lec \Bigl( \sum_{N,\, N' \ge 1} N^{2s_k} N'^{2s'} \sup_{\|v\|_{V^2_S} = 1} \Bigl|  \Bigl( \sum_{N_2 \sim N,\, N_2' \lec N'} \sum_{N_1 \ll N_2,\, N_1' \sim N'} + \sum_{N_2 \sim N,\, N_2' \gg N'} \sum_{N_1 \ll N_2,\, N_1' \lec N_2'} \\ 
 &\qquad + \sum_{N_2 \sim N,\, N_2' \sim N'} \sum_{N_1 \ll N_2,\, N_1' \ll N_2'} \Bigr) \int_{\mathbb{R}^{d+1}} \1_{[0,\, T]} n_{N_1,\, N_1'} u_{N_2,\, N_2'} \overline{v_{N,\, N'}} \, dxdt \Bigr|^2 \Bigr)^{1/2} \\ 
 &\lec M_{3,\, 1} + M_{3,\, 2} + M_{3,\, 3},  
}
where 
\EQQS{
 M_{3,\, 1}^2 &:= \sum_{N,\, N' \ge 1} N^{2s_k} N'^{2s'} \sup_{\|v\|_{V^2_S} = 1} \Bigl| \sum_{N_2 \sim N,\, N_2' \lec N'} \sum_{N_1 \ll N_2,\, N_1' \sim N'} \\ 
 &\qquad \int_{\mathbb{R}^{d+1}} \1_{[0,\, T]} n_{N_1,\, N_1'} u_{N_2,\, N_2'} \overline{v_{N,\, N'}} \, dxdt \Bigr|^2, \\ 
 M_{3,\, 2}^2 &:= \sum_{N,\, N' \ge 1} N^{2s_k} N'^{2s'} \sup_{\|v\|_{V^2_S} = 1} \Bigl| \sum_{N_2 \sim N,\, N_2' \gg N'} \sum_{N_1 \ll N_2,\, N_1' \lec N_2'} \\ 
 &\qquad \int_{\mathbb{R}^{d+1}} \1_{[0,\, T]} n_{N_1,\, N_1'} u_{N_2,\, N_2'} \overline{v_{N,\, N'}} \, dxdt \Bigr|^2, \\ 
 M_{3,\, 3}^2 &:= \sum_{N,\, N' \ge 1} N^{2s_k} N'^{2s'} \sup_{\|v\|_{V^2_S} = 1} \Bigl| \sum_{N_2 \sim N,\, N_2' \sim N'} \sum_{N_1 \ll N_2,\, N_1' \ll N_2'} \\ 
 &\qquad \int_{\mathbb{R}^{d+1}} \1_{[0,\, T]} n_{N_1,\, N_1'} u_{N_2,\, N_2'} \overline{v_{N,\, N'}} \, dxdt \Bigr|^2. 
}
From \eqref{O_431} and the estimate for $J_{4,\, 3,\, 1}$, we have 
\EQ{ \label{est_M_3}
 M_{3,\, 1} \lec \begin{cases}
                   T^{1/2} JR \qquad (d \ge 3), \\ 
                   T^{3/4} JR \qquad (d=2). 
                 \end{cases}
} 
Similarly, we can check $M_{3,\, 2}, M_{3,\, 3}$ is bounded by the right-hand side of \eqref{est_M_3}. 
%%%%% Low Schr\"odinger %%%%%
We estimate the low frequency part for the Schr\"odinger equation below. 
From Proposition \ref{inhomo_L^infty_tL^2_x} \eqref{lowfreqconti}, Proposition \ref{inhomo_local_smoothing} \eqref{lowfreqsmoothing} and Proposition \ref{inhomo_maximal_fcn} \eqref{lowfreqmaximal_2}, we have 
\EQS{ 
 &\Bigl\| P_0 \int_0^t S(t-t') (nu)(t')\, dt' \Bigr\|_{L^{\infty}_t L^2_x \, \cap \, L^{\infty}_{x_1} L^2_{x_2, ... , x_d, t} \, \cap \, L^2_{x_1} L^{\infty}_{x_2, ... , x_{d-1}, t} L^2_{x_d}} \notag \\ 
 &\lec \|P_0 (nu)\|_{L^1_{x_1} L^2_{x_2, ... , x_d, t}} \notag \\ 
 &\lec T^{1/2} \|n\|_{L^{\infty}_t L^2_{\tilde{x}} H^{s'}_{x_d}} \|u\|_{L^2_{x_1} L^{\infty}_{x_2, ... , x_{d-1}, t} L^2_{x_d}} \notag \\ 
 &\lec T^{1/2} MR.  
                  \label{low_freq_S}
}
Moreover, we have 
\EQQS{
 \Bigl\| P_0 \int_0^t S(t-t')(nu)(t')\, dt' \Bigr\|_{V^2_S} 
 &\lec \sup_{\|v\|_{U^2_S}=1} \Bigl| \int_{\mathbb{R}^{d+1}} \1_{[0,\, T]} nu \overline{P_0 v}\, dxdt \Bigr| \\ 
 &\qquad + \Bigl\| P_0 \int_0^t S(t-t')(nu)(t')\, dt' \Bigr\|_{L^{\infty}_t L^2_x}.   
}
From \eqref{low_freq_S}, the second term of the right-hand side of the above inequality is bounded by $T^{1/2} MR$. 
The first term is estimated by the H\"older inequality and the Sobolev inequality 
\EQQS{
 \Bigl| \int_{\mathbb{R}^{d+1}} \1_{[0,\, T]} nu \overline{P_0 v}\, dxdt \Bigr| 
 &\lec T \|n\|_{L^{\infty}_t L^2_x} \|u\|_{L^{\infty}_t L^2_x} \| \LR{\nabla_{\tilde{x}}}^{(d-1)/2+} \LR{\partial_{x_d}}^{1/2+} P_0 v \|_{L^{\infty}_t L^2_x} \\ 
 &\lec T \|n\|_{Z^{s_l,\, s'}_{W_{\pm}}} \|u\|_{Y^{s_k,\, s'}_S} \|P_0 v\|_{V^2_S}. 
} 
Hence by $U^2_S \hookrightarrow V^2_S$ we have 
\EQQS{
 \sup_{\|v\|_{U^2_S}=1} \Bigl| \int_{\mathbb{R}^{d+1}} \1_{[0,\, T]} nu \overline{P_0 v}\, dxdt \Bigr| \lec TJR. 
}  
%%%%% Wave %%%%%
We estimate the wave part. 
We need to estimate the following. 
\EQQS{
 \Bigl( \sum_{N_1,\, N_1' \ge 1} N_1^{2s_l} N_1'^{2s'} \sup_{\|n\|_{V^2_{W_{\pm}}} = 1} \Bigl| \int_{\mathbb{R}^{d+1}} \1_{[0,\, T]} u \bar{v} \, \overline{\omega n_{N_1,\, N_1'}}\, dxdt \Bigr|^2 \Bigr)^{1/2}. 
}
The above term is bounded by 
\EQQS{
 &\Bigl( \sum_{N_1,\, N_1' \ge 1} N_1^{2s_l} N_1'^{2s'} \sup_{\|n\|_{V^2_{W_{\pm}}} = 1} \Bigl| \Bigl( \sum_{N_2 \ll N} \sum_{N \sim N_1} + \sum_{N_2 \gec N_1} \sum_{N \sim N_2} + \sum_{N_2 \sim N_1} \sum_{N \ll N_2} \Bigr) \\ 
 &\qquad \int_{\mathbb{R}^{d+1}} \1_{[0,\, T]} u_{N_2} \overline{v_N}\, \overline{\omega n_{N_1,\, N_1'}}\, dxdt \Bigr|^2 \Bigr)^{1/2} \\ 
 &\lec R_1 + R_2 + R_3, 
}
where 
\EQQS{
 R_1^2 &:= \sum_{N_1,\, N_1' \ge 1} N_1^{2s_l} N_1'^{2s'} \sup_{\|n\|_{V^2_{W_{\pm}}} = 1} \Bigl|  \sum_{N_2 \ll N} \sum_{N \sim N_1} \int_{\mathbb{R}^{d+1}} \1_{[0,\, T]} u_{N_2} \overline{v_N}\, \overline{\omega n_{N_1,\, N_1'}}\, dxdt \Bigr|^2, \\ 
 R_2^2 &:= \sum_{N_1,\, N_1' \ge 1} N_1^{2s_l} N_1'^{2s'} \sup_{\|n\|_{V^2_{W_{\pm}}} = 1} \Bigl| \sum_{N_2 \gec N_1} \sum_{N \sim N_2} \int_{\mathbb{R}^{d+1}} \1_{[0,\, T]} u_{N_2} \overline{v_N}\, \overline{\omega n_{N_1,\, N_1'}}\, dxdt \Bigr|^2, \\
 R_3^2 &:= \sum_{N_1,\, N_1' \ge 1} N_1^{2s_l} N_1'^{2s'} \sup_{\|n\|_{V^2_{W_{\pm}}} = 1} \Bigl| \sum_{N_2 \sim N_1} \sum_{N \ll N_2} \int_{\mathbb{R}^{d+1}} \1_{[0,\, T]} u_{N_2} \overline{v_N}\, \overline{\omega n_{N_1,\, N_1'}}\, dxdt \Bigr|^2. 
}
$R_1$ is bounded by 
\EQQS{ 
 &\Bigl( \sum_{N_1,\, N_1' \ge 1} N_1^{2s_l+2} N_1'^{2s'} \sup_{\|n\|_{V^2_{W_{\pm}}}=1} \Bigl| \Bigl( 
 \sum_{N_2 \ll N,\, N_2' \ll N_1'} \sum_{N \sim N_1,\, N' \sim N_1'} 
 + \sum_{N_2 \ll N,\, N_2' \gec N_1'} \sum_{N \sim N_1,\, N' \sim N_2'} \\ 
 &\qquad 
 + \sum_{N_2 \ll N,\, N_2' \sim N_1'} \sum_{N \sim N_1,\, N' \ll N_2'} \Bigr) \int_{\mathbb{R}^{d+1}} \1_{[0,\, T]} u_{N_2,\, N_2'} \overline{v_{N,\, N'}} \overline{n_{N_1,\, N_1'}}\, dxdt \Bigr|^2 \Bigr)^{1/2} \\ 
 &\lec R_{1,\, 1} + R_{1,\, 2} + R_{1,\, 3}, 
} 
where 
\EQQS{
 R_{1,\, 1}^2 &:= \sum_{N_1,\, N_1' \ge 1} N_1^{2s_l+2} N_1'^{2s'} \sup_{\|n\|_{V^2_{W_{\pm}}}=1} \Bigl| 
 \sum_{N_2 \ll N,\, N_2' \ll N_1'} \sum_{N \sim N_1,\, N' \sim N_1'} \\ 
 &\qquad \int_{\mathbb{R}^{d+1}} \1_{[0,\, T]} u_{N_2,\, N_2'} \overline{v_{N,\, N'}} \overline{n_{N_1,\, N_1'}}\, dxdt \Bigr|^2, \\ 
 R_{1,\, 2}^2 &:= \sum_{N_1,\, N_1' \ge 1} N_1^{2s_l+2} N_1'^{2s'} \sup_{\|n\|_{V^2_{W_{\pm}}}=1} \Bigl| 
 \sum_{N_2 \ll N,\, N_2' \gec N_1'} \sum_{N \sim N_1,\, N' \sim N_2'} \\ 
 &\qquad \int_{\mathbb{R}^{d+1}} \1_{[0,\, T]} u_{N_2,\, N_2'} \overline{v_{N,\, N'}} \overline{n_{N_1,\, N_1'}}\, dxdt \Bigr|^2, \\ 
 R_{1,\, 3}^2 &:= \sum_{N_1,\, N_1' \ge 1} N_1^{2s_l+2} N_1'^{2s'} \sup_{\|n\|_{V^2_{W_{\pm}}}=1} \Bigl| 
 \sum_{N_2 \ll N,\, N_2' \sim N_1'} \sum_{N \sim N_1,\, N' \ll N_2'} \\ 
 &\qquad \int_{\mathbb{R}^{d+1}} \1_{[0,\, T]} u_{N_2,\, N_2'} \overline{v_{N,\, N'}} \overline{n_{N_1,\, N_1'}}\, dxdt \Bigr|^2.  
}
We estimate $R_{1,\, 1}$. 
From \eqref{estimate_i}, $s_k - s_l = 1/2$,   
$V^2_{W_{\pm}} \hookrightarrow L^{\infty}_t L^2_x$ and the Cauchy-Schwarz inequality we have 
\EQQS{
 R_{1,\, 1} &\lec T^{1/2} \Bigl( \sum_{N,\, N' \gec 1} N^{2(s_k+1/2)} N'^{2s'} \|v_{N,\, N'}\|_{L^{\infty}_{x_1} L^2_{x_2, ... , x_d, t}}^2 \Bigl( \| P_0 u \|_{L^2_{x_1} L^{\infty}_{x_2, ... , x_{d-1}, t} L^2_{x_d}} + \\ 
 &\quad \Bigl( \sum_{1 \le N_2 \ll N,\, 1 \le N_2' \ll N'} N_2^{2(s_k-(d-1)/2-\varepsilon)} N_2'^{2s'} \|u_{N_2,\, N_2'}\|_{L^2_{x_1} L^{\infty}_{x_2, ... , x_{d-1}, t} L^2_{x_d}}^2 \Bigr)^{1/2} \Bigr)^2 \Bigr)^{1/2} \\ 
 &\lec T^{1/2} KM. 
}
From $L^{\infty}(\mathbb{R}) \hookrightarrow H^s(\mathbb{R})$ with $s > 1/2$ for $x_d$ variable, $R_{1,\, 2}, R_{1,\, 3}$ are treated similar to the case $R_{1,\, 1}$. 
We estimate $R_{2,\, 1}$. 
For $d \ge 3$, from Lemma \ref{trilinear} \eqref{3dim_1}, $V^2_{W_{\pm}} \hookrightarrow L^{\infty}_t L^2_x$ and the Cauchy-Schwarz inequality we have 
\EQQS{ 
 R_{2,\, 1} 
 &= \Bigl( \sum_{N_1,\, N_1' \ge 1} N_1^{2s_l+2} N_1'^{2s'} \sup_{\|n\|_{V^2_{W_{\pm}}}=1} \Bigl| \sum_{N_2 \gec N_1,\, N_2' \ll N_1'} \sum_{N \sim N_2,\, N' \sim N_1'} \\ 
 &\qquad \int_{\mathbb{R}^{d+1}} \1_{[0,\, T]} u_{N_2,\, N_2'} \overline{v_{N,\, N'}} \overline{n_{N_1,\, N_1'}}\, dxdt \Bigr|^2 \Bigr)^{1/2} \\ 
 &\lec T^{1/2} \sum_{N_2 \gec 1,\, N_2'} \sum_{N \sim N_2,\, N' \gec 1} \Bigl( \sum_{N_1 \lec N_2,\, N_1' \sim N'} N_1^{2(s_k + 1/2 + (d-3)/2 + )} N_1'^{2s'} N_2'^{1+} \\ 
 &\qquad \|u_{N_2,\, N_2'}\|_{V^2_S}^2 \|v_{N,\, N'}\|_{V^2_{W_{\pm}}}^2 \Bigr)^{1/2} \\ 
 &\lec T^{1/2} \|u\|_{Y^{s_k,\, s'}_S} \|v\|_{Y^{s_k,\, s'}_S} 
 \lec T^{1/2} J^2. 
}
For $d=2$, Lemma \ref{trilinear} \eqref{2d_1}, $V^2_{W_{\pm}} \hookrightarrow L^{\infty}_t L^2_x$, the Cauchy-Schwarz inequality leads   
\EQQS{
 R_{2,\, 1} &\lec T^{3/4} \sum_{N_2 \gec 1,\, N_2'} \sum_{N \sim N_2,\, N' \gec 1}  \Bigl( \sum_{N_1 \lec N_2,\, N_1' \sim N'} N_1^{2(s_k+1/2)} N_1'^{2s'} N_2'^{1+} \\  
 &\qquad \|u_{N_2,\, N_2'}\|_{V^2_S}^2 \|v_{N,\, N'}\|_{V^2_S}^2 \Bigr)^{1/2} \\ 
 &\lec T^{3/4} \|u\|_{Y^{s_k,\, s'}_S} \|v\|_{Y^{s_k,\, s'}_S} 
 \lec T^{3/4} J^2. 
}
Similar to the estimate for $R_{2,\, 1}$, we can obtain the estimate for $R_{2,\, 2}, R_{2,\, 3}$. 
By symmetry, the estimate for $R_3$ is obtained in a similar way to $R_1$. 
%%%%% Low wave %%%%%
Finally we estimate the low frequency part of the Duhamel term for the wave equation. 
\EQS{
 \Bigl\| P_0 \int_0^t W_{\pm}(t-t') \omega (u \bar{v})(t')\, dt' \Bigr\|_{U^2_{W_{\pm}}} 
 \lec \sup_{\|n\|_{V^2_{W_{\pm}}}=1} \Bigl| \int_{\mathbb{R}^{d+1}} \1_{[0,\, T]} u\bar{v} \overline{P_0 \omega n}\, dxdt \Bigr|. 
          \label{low_wave} 
}
By the H\"older inequality and $\|u\|_{L^4_t L^{2(d-1)/(d-2)}_{\tilde{x}} L^2_{x_d}} \lec \|u\|_{Y^{s_k,\, s'}_S}$, the right-hand side of \eqref{low_wave} is bounded by  
\EQQS{
 &T^{1/2} \sup_{\|n\|_{V^2_{W_{\pm}}}=1} \|u\|_{L^4_t L^{2(d-1)/(d-2)}_{\tilde{x}} L^2_{x_d}} \|v\|_{L^4_t L^{2(d-1)/(d-2)}_{\tilde{x}} L^2_{x_d}} %\|\LR{\nabla_{\tilde{x}}}^{(d-3)/2+} \LR{\nabla_{x_d}}^{1/2+} P_0 \omega n\|_{L^{\infty}_t L^2_x} 
 \|P_0 \omega n\|_{L^{\infty}_t L^{d-1}_{\tilde{x}} L^{\infty}_{x_d}} \\ 
% \|P_0 n\|_{L^{\infty}_t L^2_x} \\ 
 &\lec T^{1/2}\|u\|_{Y^{s_k,\, s'}_S} \|v\|_{Y^{s_k,\, s'}_S} 
 \lec T^{1/2} J^2 
}
for $d \ge 3$. 
For $d=2$, we see 
\EQQS{
 &\sup_{\|n\|_{V^2_{W_{\pm}}}=1} \Bigl| \int_{\mathbb{R}^3} \1_{[0,\, T]} u\bar{v} \overline{P_0 \omega n}\, dxdt \Bigr| \\ 
 &\lec T^{3/4} \sup_{\|n\|_{V^2_{W_{\pm}}}=1} \|u\|_{L^4_t L^{\infty}_{x_1} L^2_{x_2}} \|v\|_{L^{\infty}_t L^2_x} \| \LR{\partial_{x_2}}^{1/2+} P_0 \omega n \|_{L^{\infty}_t L^2_x} \\ 
 &\lec T^{3/4} \|u\|_{Y^{s_k,\, s'}_S} \|v\|_{Y^{s_k,\, s'}_S} 
 \lec T^{3/4} J^2.  
}

\end{proof}

Therefore for $0 < T < 1$, we have 
\EQQS{
 &\|\mathcal{G}u\|_{X_S} \lec \|u_0\|_{H^{s_k,\, s'}_x} + T^{1/2} \|n\|_{Z^{s_l,\, s'}_{W_{\pm}}} \|u\|_{X_S}, \\ 
 &\|\tilde{\mathcal{G}}n\|_{Z^{s_l,\, s'}_{W_{\pm}}} \lec \|n_{\pm 0}\|_{H^{s_l,\, s'}_x} + T^{1/2} \|u\|_{X_S}^2, \\ 
 &\|\mathcal{G}u - \mathcal{G}v\|_{X_S} \lec T^{1/2} (\|n-m\|_{Z^{s_l,\, s'}_{W_{\pm}}}\|u\|_{X_S} + \|m\|_{Z^{s_l,\, s'}_{W_{\pm}}}\|u-v\|_{X_S}), \\ 
 &\|\tilde{\mathcal{G}}n - \tilde{\mathcal{G}}m\|_{Z^{s_l,\, s'}_{W_{\pm}}} \lec T^{1/2} \|u-v\|_{X_S}(\|u\|_{X_S} + \|v\|_{X_S}).  
}
This shows the existence and the uniqueness of a local solution $u, n_{\pm}$ in $X_S, Z^{s_l,\, s'}_{W_{\pm}}$ with $T=T(\|u_0\|_{H^{s_k,\, s'}_x}, \|n_{\pm 0}\|_{H^{s_l,\, s'}_x})$ small enough.

\section{Appendix} 
In this section we show the bilinear Strichartz estimate, namely Proposition \ref{x01/2}. 
For the Zakharov system, see \cite{BHHT} and \cite{BH}. 
For dyadic numbers $N,\, N',\, L$, we set 
\EQQS{ 
 &\mathcal{P}_{N,\, N'} := \{ (\xi,\, \xi') \in \mathbb{R}^{d-1} \times \mathbb{R}\, |\, N/2 \le |\xi| \le 2N,\, N'/2 \le |\xi'| \le 2N' \}, \\
 &\mathcal{P}_{N,\, 0} := \{ (\xi,\, \xi') \in \mathbb{R}^{d-1} \times \mathbb{R}\, |\, N/2 \le |\xi| \le 2N,\, |\xi'| \le 2 \}, \\
 &\mathcal{P}_{0,\, N'} := \{ (\xi,\, \xi') \in \mathbb{R}^{d-1} \times \mathbb{R}\, |\, |\xi| \le 2,\, N'/2 \le |\xi'| \le 2N' \}, \\
 &\mathcal{P}_0 = \mathcal{P}_{0,\, 0} := \{ (\xi,\, \xi') \in \mathbb{R}^{d-1} \times \mathbb{R}\, |\, |\xi| \le 2,\, |\xi'| \le 2 \}, \\
 &\mathcal{W}_L^{\pm} := \{ (\tau,\, \xi,\, \xi') \in \mathbb{R} \times \mathbb{R}^{d-1} \times \mathbb{R} \, |\, L/2 \le \bigl|\tau \pm |\xi| \bigr| \le 2L \}, \\
 &\mathcal{W}_0^{\pm} := \{ (\tau,\, \xi,\, \xi') \in \mathbb{R} \times \mathbb{R}^{d-1} \times \mathbb{R} \, |\, \bigl|\tau \pm |\xi| \bigr| \le 2 \}, \\
 &\mathcal{S}_L := \{ (\tau,\, \xi,\, \xi') \in \mathbb{R} \times \mathbb{R}^{d-1} \times \mathbb{R} \, |\, L/2 \le \bigl|\tau + |\xi|^2 + \xi' \bigr| \le 2L \}, \\ 
 &\mathcal{S}_0 := \{ (\tau,\, \xi,\, \xi') \in \mathbb{R} \times \mathbb{R}^{d-1} \times \mathbb{R} \, |\, \bigl|\tau + |\xi|^2 + \xi' \bigr| \le 2 \}.  
}

\begin{prop}   \label{x01/2} 
Let $d \ge 2$. $\rm(\hspace{.18em}i-a\hspace{.18em})$ Let $u, v \in L^2(\mathbb{R}^{1+d})$ be such that 
\EQQS{
 \supp \F u \subset \mathcal{W}_{L_1}^{\pm} \cap \bigl( \mathbb{R} \times ((C \times \mathbb{R}) \cap \mathcal{P}_{N_1,\, N_1'})\bigr),\qquad \supp \F v \subset \mathcal{S}_{L_2} \cap (\mathbb{R} \times \mathcal{P}_{N_2,\, N_2'}) 
} 
for dyadic numbers $L_i,\, N_i,\, N_i'\ (i=1,\, 2)$ and a cube $C \subset \mathbb{R}^{d-1}$ of side length $e$. 
If $N_1 \lec N_2,\, N_2 \gg 1$ and $N_1' \gec N_2'$, it holds that  
\EQQS{
 \|uv\|_2 \lec N_2^{-1/2}N_1^{(d-2)/2}N_2'^{1/2} L_1^{1/2} L_2^{1/2} \|u\|_2 \|v\|_2.     
}
$\rm(\hspace{.18em}i-b\hspace{.18em})$ Let $u, v \in L^2(\mathbb{R}^{1+d})$ be such that 
\EQQS{
 \supp \F u \subset \mathcal{W}_{L_1}^{\pm} \cap \bigl( \mathbb{R} \times ((C \times \mathbb{R}) \cap \mathcal{P}_{N_1,\, N_1'})\bigr),\qquad \supp \F v \subset \mathcal{S}_{L_2} \cap (\mathbb{R} \times \mathcal{P}_{N_2,\, N_2'}) 
} 
for dyadic numbers $L_i,\, N_i,\, N_i'\ (i=1,\, 2)$ and a cube $C \subset \mathbb{R}^{d-1}$ of side length $N$.  
If $N_1 \sim N_2 \gg N,\, N_2 \gg 1$ and $N_1' \sim N_2' \gec N'$, it holds that  
\EQQS{
 \|P_{N,\, N'} (uv)\|_2 \lec N_2^{-1/2}N^{(d-2)/2}N'^{1/2} L_1^{1/2} L_2^{1/2} \|u\|_2 \|v\|_2.     
}
$\rm(\hspace{.08em}ii\hspace{.08em})$ Let $u, v \in L^2(\mathbb{R}^{1+d})$ be such that 
\EQQS{
 \supp \F u \subset \mathcal{S}_{L_3} \cap (\mathbb{R} \times \mathcal{P}_{N_3,\, N_3'}), \qquad \supp \F v \subset \mathcal{S}_{L_4} \cap (\mathbb{R} \times \mathcal{P}_{N_4,\, N_4'}) 
}
for dyadic numbers $L_i,\, N_i,\, N_i'\ (i=3,\, 4)$. 
If $N_3 \ll N_4,\, N_4 \gg 1$ and $N_3' \lec N_4'$, it holds that 
\EQQS{
 \|uv\|_2 \lec N_4^{-1/2} N_3^{(d-2)/2} N_3'^{1/2} L_3^{1/2} L_4^{1/2} \|u\|_2 \|v\|_2.  
}
\end{prop}

\begin{proof}
Let $f := \hat{u},\, g := \hat{v}$. 
By the Cauchy-Schwarz inequality, we have 
\EQQS{
 &\Bigl\| \int f(\tau_1,\, \xi_1,\, \xi_1')\, g(\tau - \tau_1,\, \xi - \xi_1,\, \xi'-\xi_1') \, d\tau_1 d\xi_1 d\xi_1' \Bigr\|_{L^2_{\tau,\, \xi,\, \xi'}} \\
 &\lec \sup_{\tau,\, \xi,\, \xi'} |E(\tau,\, \xi,\, \xi')|^{1/2}\, \|f\|_2 \|g\|_2 
}
where 
\EQQS{
 E(\tau,\, \xi,\, \xi') = \{ (\tau_1,\, \xi_1,\, \xi_1') \in \supp f \, |\, (\tau - \tau_1,\, \xi - \xi_1,\, \xi'-\xi_1') \in \supp g \} \subset \mathbb{R}^{1+d}.
}
Put $\underline{l} := \min\{ L_1,\, L_2 \},\, \ol{l} := \max\{ L_1,\, L_2 \}$. 
By the Fubini theorem, 
\EQQS{
 |E(\tau,\, \xi,\, \xi')| 
   &\le \underline{l}\, N_2' \, 
        \bigl|  \bigl\{ \xi_1 \, \bigl|\, \bigl| \tau \pm |\xi_1| + |\xi - \xi_1|^2 + \xi' \bigr| \lec \ol{l},\, \xi_1 \in C,\, |\xi_1| \sim N_1,\, \\
 &\hspace{20mm}         |\xi_1'| \sim N_1',\, |\xi - \xi_1| \sim N_2,\, |\xi'-\xi_1'| \sim N_2' \bigr\} \bigr|.
}
In the right-hand side of the above inequality, the subset of the $\xi_1$ is contained in a cube of side length $m$, where $m \sim \min\{e,\, N_1\} \sim N_1$. 
For some $i \in \{1, ... , d-1\}$, we set $|(\xi-\xi_1)_i| \gec N_2$, 
where $(\xi-\xi_1)_i$ denotes the $i$-th component of $\xi - \xi_1$. 
We compute  
\EQS{
 |\partial_{\xi_{1,\, i}}(\tau \pm |\xi_1| + |\xi-\xi_1|^2 + \xi')| 
 = \Bigl| \pm \frac{\xi_{1,\, i}}{|\xi_1|} -2(\xi - \xi_1)_i \Bigr|,     \label{Mod}
}
where $\xi_{1,\, i}$ be the $i$-th component of $\xi_1$. 
Since $|\xi_{1,\, i}| \le |\xi_1|$ and $|(\xi - \xi_1)_i| \gec N_2$, 
\EQQS{
 (R.H.S.\ of\ \eqref{Mod}) \gec N_2. 
} 
Therefore,  
\EQS{
  |\partial_{\xi_{1,\, i}}(\tau \pm |\xi_1| + |\xi-\xi_1|^2 + \xi')| \gec N_2.   \label{deriv}
}
Hence by \eqref{deriv} and the mean value theorem, we have 
\EQQS{
 &\bigl| \bigl\{ \xi_1 \, \bigl|\, \left| \tau \pm |\xi_1| + |\xi - \xi_1|^2 + \xi' \right| \lec \ol{l},\, \xi_1 \in C,\, |\xi_1| \sim N_1, |\xi'| \sim N_1',\, \\
 &\hspace{20mm} |\xi - \xi_1| \sim N_2,\, |\xi' - \xi_1'| \sim N_2' \bigr\} \bigr| \\
 &\lec N_2^{-1}\, m^{d-2}\, \ol{l}. 
}
From $m \sim N_1$, we have 
\EQQS{
 |E(\tau ,\, \xi ,\, \xi')|^{1/2} 
 \lec \underline{l}^{1/2} N_2'^{1/2} N_2^{-1/2} m^{(d-2)/2}\, \ol{l}^{1/2} 
 \sim N_2^{-1/2} N_1^{(d-2)/2} N_2'^{1/2} L_1^{1/2} L_2^{1/2}. 
}
Thus, we obtain $\rm(\hspace{.18em}i-a\hspace{.18em})$. 
$\rm(\hspace{.18em}i-b\hspace{.18em})$ is proved by the same manner as for $\rm(\hspace{.18em}i-a\hspace{.18em})$, hence we omit the proof.   
$\rm(\hspace{.08em}ii\hspace{.08em})$ follows from the similar argument as the estimate for the case $\rm(\hspace{.18em}i-a\hspace{.18em})$. 
Indeed in this case, we estimate $|E(\tau,\, \xi,\, \xi')|$ as follows. 
\EQQS{
 |E(\tau,\, \xi,\, \xi')| 
 &\le \underline{l} \, N_3'\, \bigl| \bigl\{ \xi_3 \, \bigl|\, \bigl| \tau + |\xi-\xi_3|^2 + |\xi_3|^2 + \xi' \bigr| \lec \ol{l},\, |\xi_3| \sim N_3,\, \\ 
 &\hspace{20mm} |\xi_3'| \sim N_3',\, |\xi - \xi_3| \sim N_4,\, |\xi'-\xi_3'| \sim N_4' \bigr\} \bigr|
}
where $\underline{l} := \min\{ L_3,\, L_4\},\, \ol{l} := \max\{ L_3,\, L_4\}$.   
For some $i \in \{ 1, ... , d-1\}$, we set $|(\xi - \xi_3)_i| \gec N_4$ where $(\xi - \xi_3)_i$ denotes the $i$-th component of $\xi - \xi_3$. 
Then by $N_4 \gg N_3$,  we have 
\EQQS{    
 |\partial_{\xi_{3,\, i}}(\tau + |\xi - \xi_3|^2 + |\xi_3|^2 + \xi')| 
 = | -2(\xi - \xi_3)_i + 2\xi_{3,\, i}| 
 \gec N_4, 
}
where $\xi_{3,\, i}$ be the $i$-th component of $\xi_3$. 
Thus, we have 
\EQQS{
 &\bigl| \bigl\{ \xi_3 \, \bigl|\, \bigl| \tau + |\xi-\xi_3|^2 + |\xi_3|^2 + \xi' \bigr| \lec \ol{l},\, |\xi_3| \sim N_3,\, \\ 
 &\hspace{20mm} |\xi_3'| \sim N_3',\, |\xi - \xi_3| \sim N_4,\, |\xi'-\xi_3'| \sim N_4' \bigr\} \bigr| 
 \lec N_4^{-1} N_3^{d-2} \ol{l}. 
}
Therefore, we obtain the desired result. 
\end{proof}

\begin{rem} \label{BilinearStrichartz}
 If we assume $N_1 \gg N_2$ instead of $N_1 \lec N_2$ in Proposition \ref{x01/2} $\rm(\hspace{.18em}i-a\hspace{.18em})$, it holds that 
\EQQS{ 
 \|uv\|_2 \lec N_2^{(d-3)/2} N_2'^{1/2} L_1^{1/2} L_2^{1/2}\|u\|_2 \|v\|_2.           
} 
The proof of the above inequality is obtained by the same way, hence we omit it.  
\end{rem}

\section*{Acknowledgement}
The author would like to appreciate the anonymous referee who pointed out the problem of my earlier  paper, which is used $X^{s,\, b}$ space.   
The author is supported by JSPS KAKENHI Grant Number 820200500051.


\begin{thebibliography}{10}
\bibitem{BL}V. Barros and F. Linares, {\itshape A Remark on the well-posedness of the degenerate Zakharov system}, Comm. Pure. Appl. Anal. {\bfseries 14} (2015), 1259--1274. 
\bibitem{BGHN}I. Bejenaru, Z. Guo, S. Herr and K. Nakanishi, {\itshape Well-posedness and scattering for the Zakharov system in four dimensions}, Anal. PDE. {\bfseries 8} (2015), 2029--2055. 
\bibitem{BHHT}I. Bejenaru, S. Herr, J. Holmer and D. Tataru, {\itshape On the 2D Zakharov system with $L^2$ Schr\"{o}dinger data}, Nonlinearity {\bfseries 22} (2009), 1063--1089.
\bibitem{BH}I. Bejenaru and S. Herr, {\itshape Convolutions of singular measures and applications to the Zakharov system}, J. Funct. Anal. {\bfseries 261} (2011), 478--506.
\bibitem{Bo}J. Bourgain, {\itshape Fourier transform restriction phenomena for certain lattice subsets and application to nonlinear evolution equations I. Schr\"{o}dinger equations}, GAFA {\bfseries 3} (1993), 107--156. 
\bibitem{CHN}T. Candy, S. Herr and K. Nakanishi, {\itshape The Zakharov system in dimension $d \ge 4$}, arXiv: 1912.05820v1 (2019), 35pp. 
\bibitem{CC}M. Colin and T. Colin, {\itshape On a quasilinear Zakharov system describing laser-plasma interactions}, Differential Integral Equations {\bfseries 17} (2004), 297--330. 
\bibitem{CM}T. Colin and G. M\'{e}tivier, {\itshape Instabilities in Zakharov  equations for lazer propagation in a plasma. Phase space
analysis of partial differential equations}, Progr. Nonlinear Differential Equations Appl. {\bfseries 69} (2006), 63--81, Birkh\"{a}user Boston, MA. 
\bibitem{GTV}J. Ginibre, Y. Tsutsumi and G. Velo, {\itshape On the Cauchy problem for the Zakharov system}, J. Funct. Anal. {\bfseries 151} (1997), 
no. 2, 384--436. 
\bibitem{HHK}M. Hadac, S. Herr, and H. Koch, {\itshape Well-posedness and scattering for the KP-II equation in a critical space}, Ann. I. H. Poincar\'{e} AN {\bfseries 26} (2009), 917--941. 
\bibitem{HHK2}M. Hadac, S. Herr, and H. Koch, {\itshape Erratum to "Well-posedness and scattering for the KP-II equation in a critical space"[Ann. I. H. Poincar\'{e} AN {\bfseries 26} (2009), 917--941]}, Ann. I. H. Poincar\'{e} AN {\bfseries 27} (2010), no. 3, 971--972.  
\bibitem{KaT}I. Kato and K. Tsugawa, {\itshape Scattering and well-posedness for the Zakharov system at a critical space in four and more spatial dimensions}, Differerential Integral Equations. {\bfseries 30} (2017), no.9-10, 763--794. 
\bibitem{KPV}C. Kenig, G. Ponce and L. Vega, {\itshape A bilinear estimate with applications to the KdV equation}, J. Amer. Math. Soc. {\bfseries 9} (1996), no.2, 573--603. 
\bibitem{KPV2}C. Kenig, G. Ponce and L. Vega, {\itshape Well-posedness of the initial value problem for the Korteweg-de Vries equation}, J. Amer. Math. Soc. {\bfseries 4} (1991), no.2, 323--347. 
\bibitem{LPS}F. Linares, G. Ponce and J-C. Saut, {\itshape On a degenerate Zakharov system}, Bull. Braz. Math. Soc. (N.S.) {\bfseries 36} (2005), no. 1, 1--23. 
\bibitem{MR}L. Molinet and F. Ribaud, {\itshape Well-posedness results for the generalized Benjamin-Ono equation with small initial data}, J. Math. Pures Appl. (9), {\bfseries 83} (2004), no. 2, 277--311. 
\bibitem{SulSul}C. Sulem and P-L. Sulem, {\itshape The nonlinear Schr\"{o}dinger Equation: Self-Focusing and Wave Collapse}, Applied Mathematical Sciences 139, Springer, (1999). 
\bibitem{Ta}D. Tataru, {\itshape Local and global results for wave maps I}, Comm. Part. Diff. Eq. {\bfseries 23} (1998), 1781--1793.
\end{thebibliography}
\end{document}